\documentclass[13pt,oneside,reqno]{amsart}
\usepackage{geometry}                
\usepackage{graphicx}
\usepackage{amssymb}
\usepackage{amsmath}
\usepackage{mathrsfs}
\usepackage{epstopdf}
\usepackage{epstopdf}
\usepackage{graphicx}
\usepackage{caption}
\usepackage{subcaption}
\usepackage{epstopdf}
\usepackage{fancyref}
\usepackage{bbm}
\usepackage{tikz}
\usepackage{enumitem}
\usepackage{calc}
\usepackage{anyfontsize}
\usepackage[symbol]{footmisc}

\newcommand{\pd}[2]{\frac{\partial{#1}}{\partial{#2}}}

\DeclareMathOperator{\R}{\mathbb{R}}
\DeclareMathOperator{\C}{\mathbb{C}}
\DeclareMathOperator{\Q}{\mathbb{Q}}
\DeclareMathOperator{\Z}{\mathbb{Z}}

\theoremstyle{definition} 
\newtheorem*{definition}{Definition}
\newtheorem{thm}{Theorem}
\newtheorem{lemma}{Lemma}
\newtheorem{cor}{Corollary}
\newtheorem{prop}{Proposition}
\newtheorem*{theorem}{Theorem}

\theoremstyle{remark}
\newtheorem*{remark}{Remark}
\newtheorem*{notation}{Notation}
\newtheorem*{note}{Note}
\newtheorem*{ex}{Example}

\newtheorem*{claim}{Claim}

\newcommand{\CP}[1]{\mathrm{\C\!\text{P}}^#1}
\numberwithin{equation}{section}

\title{The Stabilized Symplectic Embedding Problem for Polydiscs}
\author[Daniel Irvine]{Daniel Irvine$^{\dagger}$}\thanks{$\dagger\,$ The author was supported in part by NSF RTG grant 1045119. }
\address{Department of Mathematics, University of Michigan, Ann Arbor, MI 48109-1043, USA.}
\email{dirvine@umich.edu}

\begin{document}
\maketitle \vspace{-10mm}
\begin{abstract} A stabilized polydisc is a product of a symplectic polydisc and several copies of the complex plane.
This paper gives a complete characterization of symplectic embeddings of stabilized polydiscs into other stabilized polydiscs.
\end{abstract}

\section{Introduction}\label{intro}
Let us begin with $\C^n \cong \R^{2n}$ having coordinates $x_1,y_1,...,x_n,y_n$ and having been equipped with the standard symplectic form
\[ \omega = \sum_{i=1}^n dx_i\wedge dy_i.\]
We define the \textit{symplectic polydisc of capacities $r_1$ through $r_n$} to be the set{\footnote[3]{\hspace{3mm}In this context, the word capacity should not be interpreted as in symplectic capacity.}}
\[ P(r_1,...,r_n) = \left\{ (x_1,y_1,...\, ,x_n,y_n) \in \mathbb{R}^{2n} \left|\,\, \pi \cdot\left( x_j^2 + y_j^2\right) \leq r_j \, , \, j = 1,2,...,n \right. \right\},\]
with the induced symplectic form. A \textit{stabilized polydisc} is the symplectic manifold $P(r_1,r_2) \times \R^{2n-4}$ with the symplectic form inherited from $\R^{2n}$ viewed as a product symplectic manifold.
Notice that in this definition we always use a 4-dimensional polydisc in order to make the resulting stabilized polydisc a manifold of dimension $2n$, $n\geq 3$.

In a similar way we define a \textit{symplectic ellipsoid of capacities $r_1$ through $r_n$} to be the set 
 \[ E(r_1,...,r_n) = \left\{ (x_1,y_1,...\, ,x_n,y_n) \in \mathbb{R}^{2n} \left| \sum_{j=1}^{n} \frac{\pi\cdot\left( x_j^2 + y_j^2\right)}{r_j} \leq1 \right. \right\},\]
with the induced symplectic form. In each of the above definitions, we make it a convention to order the capacities as $r_1 \leq r_2 \leq ... \leq r_n$, because a permutation of the coordinates $(x_i,y_i)$ would be a linear symplectomorphism. 
%
 \begin{notation}\label{hook}
The symbol $U \hookrightarrow V$ means that there exists a symplectic embedding $K \to V$ for every compact set $K \subset Interior(U)$. Take, for instance, $K = (1-\epsilon)U$ for any $0< \epsilon < 1$, and for $U$ compact.
\end{notation}

This will be the class of embeddings that is the focus of this paper. 

We shall prove the following main theorem about stabilized polydiscs. 
\begin{theorem}(Main theorem) 
Suppose  that $n\geq 3$. There exists a symplectic embedding \[P(1,x)\times\R^{2n-4} \hookrightarrow P(a,b) \times \R^{2n-4}\] if and only if either
\begin{itemize}
\item $a \geq 2$, or
\item $1 \leq a < 2$ and $b \geq x$.
\end{itemize}
\end{theorem}
\begin{proof} At first we will focus on the case $n=3$. At the conclusion of this paper (section \ref{highdim}), we will  discuss how all of the proofs generalize to higher dimensions. So for now, we fix $n=3$.

 Let us immediately dispense with the easy cases of this theorem. The ``if'' ($\Longleftarrow$) portion of this theorem follows from the work of Hind and Kerman (\cite{newobs} $\S$4). It was later shown in \cite{optemb} that these embeddings are possible using symplectic folding in the case $a>2$. Next, for the ``only if'' claim ($\Longrightarrow$), we assume that a symplectic embedding, of the sort mentioned in the theorem, exists.   If the capacity $a$ happens to satisfy $a \geq 2$, then the first bullet point of the conclusion is proved. Otherwise we must show if $1\leq a < 2$ then $b \geq x$ to prove the second bullet point of the conclusion. We prove this claim by contradiction, and this will be the content of the rest of the paper.

Next, we explain why it suffices to prove the theorem in the case $x>2$. We assume that the theorem holds for $x>2$ and extend the result to $1\leq x \leq 2$ in the following five claims. In each of these claims, the ``if'' ($\Longleftarrow$) statement of the theorem will be immediate. We prove the ``only if'' ($\Longrightarrow$) statement in each case.

\textbf{Case 1.} The main theorem holds when $x=2$.

To prove case 1, notice that when $x =2$ and $a \geq 2$, then there is no contradiction to the theorem statement. If $x=2$ and $1\leq a < 2$, then we are assuming
 \[ P(1,2) \times \R^2 \hookrightarrow P(a,b) \times \R^2.\]
 For $0 < \lambda < 1$, inclusion clearly gives
 \[ P(1 - \lambda, 2) \times \R^2 \hookrightarrow P(a,b)\times \R^2.\]
 Re-scaling the polydiscs by $\frac{1}{1-\lambda}$ gives
 \[ P\left( 1, {2}/({1-\lambda})\right) \times \R^2 \hookrightarrow P(a/(1-\lambda), b/(1-\lambda)) \times \R^2.\]
 Because $a<2$, we can choose $\lambda$ so that ${a}/({1-\lambda})<2$. Then we have reduced the case of $x=2$ to the case of $x>2$.

\textbf{Case 2.} Assume $a=b$ and $1\leq x \leq 2$. We claim that there exists a symplectic embedding 
\[ P(1,x)\times \R^2 \hookrightarrow P(b,b) \times \R^2\] 
if and only if $b\geq x$. 

To prove case 2,  we assume, towards a contradiction, that there exists an embedding 
\[ P(1,x)\times \R^2 \hookrightarrow P(b,b) \times \R^2\]  with $b<x$. Then, by scaling, there would exist exists an embedding 
\[ P(2/x,2)\times \R^2 \hookrightarrow P((2b)/x,(2b)/x) \times \R^2.\]
By inclusion, there must also exist an embedding
\[ P(1,2)\times \R^2 \hookrightarrow P((2b)/x,(2b)/x) \times \R^2,\]
but since $(2b)/x < 2$, this would contradict case 1. This completes the proof of case 2.

\textbf{Case 3.} Assume $b \geq 2$ and $1 \leq x \leq 2$. We claim there exists a symplectic embedding 
\[ P(1,x)\times \R^2 \hookrightarrow P(\lambda, \lambda b) \times \R^2\] 
if and only if $\lambda \geq 1$.

Again, the ``if'' statement is immediate. To prove the ``only if'' statement by contradiction, we assume that there is an embedding with $\lambda < 1$. An application of Gromov non-squeezing (or the first Ekland-Hofer capacity) to the supposed embedding would imply that $1 \leq \lambda$. This creates a contradiction, and completes the proof of case 3.

\textbf{Case 4.} Assume $1\leq b \leq 2$ and $b \leq x \leq 2$. We claim there exists a symplectic embedding 
\[ P(1,x)\times \R^2 \hookrightarrow P(\lambda, \lambda b) \times \R^2\] 
if and only if $\lambda \geq x/b$.

To prove the ``only if'' statement by contradiction, assume there is an embedding with $\lambda < x/b$.  Then, by scaling, we have an embedding.
\[ P(2/x,2)\times \R^2 \hookrightarrow P(2\lambda/x, 2\lambda b/x) \times \R^2.\]
In particular, there would be an embedding 
\[ P(1,2)\times \R^2 \hookrightarrow P(2\lambda/x, 2\lambda b/x) \times \R^2.\] 
If $2\lambda/x < 1$, this contradicts Gromov non-squeezing. If $2\lambda /x  \geq 1$, this contradicts the main theorem above. Either way, this proves case 4.

\textbf{Case 5.} Assume $1\leq b \leq 2$ and $1 \leq x \leq b$. We claim there exists a symplectic embedding 
\[ P(1,x)\times \R^2 \hookrightarrow P(\lambda, \lambda b) \times \R^2\] 
if and only if $\lambda \geq 1$. 

The ``only if'' statement follows immediately from Gromov non-squeezing. This completes case 5.

\end{proof}

 Recaptitulating, for the remainder of this paper, we shall assume that $x>2$, that \[P(1,x)\times\R^{2} \hookrightarrow P(a,b) \times \R^{2}\]  
 and that $1\leq a < 2$, but $b < x$. We shall prove the main theorem by contradiction.

%
%

We can also use the main theorem to obstruct embeddings of four-dimensional polydiscs, as the following result illustrates.
\begin{cor}\label{4di} Suppose that $P(1,x)$ symplectically embeds into $P(a,b)$. Then, by taking a product with the identity map, we find that $P(1,x)\times \R^2 \hookrightarrow P(a,b)\times \R^2$. Assume further that $x\geq 2$. Then the main theorem implies that either $a\geq 2$ or $1 \leq a < 2$ and $b \geq x$. 
\end{cor}

We compare this corollary to a result of Hutchings, which has been modified to match the notation of this paper. 

\begin{thm}\label{4dhutch}(Hutchings, \cite{beyond}) Suppose that $P(1,x)$ symplectically embeds into $P(a,b)$ with $1 \leq a\leq b$. Assume further that
\begin{equation}\label{hutch} 1\leq x \leq \frac{2(b/a)}{1+\frac{(b/a)-1}{4\left\lceil \frac{b}{a}\right\rceil -1}} \quad ( \geq 2).
\end{equation}
Then $b \geq x$. 
\end{thm}

The upper bound $(\ref{hutch})$ is always at least $2$. So we may as well assume $1\leq x \leq 2$, and conclude $b \geq x$. Together, Theorem \ref{4dhutch} and Corollary \ref{4di} give a full picture of four-dimensional embeddings $P(1,x)$ into $P(a,b)$ for all values of $x$. The two results reach similar conclusions.

Specifically, let us define a four-dimensional embedding function $f$, in the spirit of \cite{stair}, also known as an embedding capacity function. For $x\geq 1, a\geq 1$, we set
\[ f(x,a) := \inf_{b}\left\{ \text{there exists a symplectic embedding } P(1,x) \text{ into } P(a,b), \text { with } b\geq a\right\}.\]

The two preceding results show that for $1\leq a < 2$, $f(x,a) = x$. The case of $a \geq 2$ is only partially known because this is the regime of symplectic folding (see proposition 4.4.4 and Figure 7.2 in \cite{Schlenk}).

An outline of the sections of this paper follows. In section \ref{ell} we examine the Reeb dynamics of a symplectic ellipsoid, and write down a formula for the generalized Conley-Zehnder index. In section \ref{polyd} we compute the generalized Conley-Zehnder indices for Reeb orbits on a polydisc, and describe a special neighborhood of a Lagrangian torus, along which we shall \textit{stretch the neck}. In section \ref{neck}, we assert the existence of a holomorphic curve in a compactification of $P(a,b)\times \R^2$ with negative end asymptotic to a skinny ellipsoid. We then stretch the neck, and in sections \ref{fred} and \ref{special} we examine the geometry of the aforementioned curve. The moral of this story is that the energy of a holomorphic curve must be distributed in a precise way. In section \ref{curveproof} we prove the existence theorem of the holomorphic curve from section \ref{neck}. The proof in section \ref{curveproof} can be read before reading all of section \ref{neck}, if so desired. Finally, in section \ref{highdim} we explain how all of the results of this paper extend to dimension $2n\geq 6$.

In proving the main theorem, we will introduce many parameters. For easy reference, let us define these parameters upfront in a list that has some interdependence:
\begin{itemize}
\item $n$ will always denote half the (real) dimension of the ambient symplectic manifold.
\item $x$, $a$, and $b$ are capacities of the polydiscs in the statement of the main theorem.
\item $d$ is an integer parameter satisfying $2d+1 >> b$.
\item $S$ is a very large real number satisfying $da+b << S$. Because of the definition of $\hookrightarrow$, we can and shall replace $P(1,x)\times \mathbb{R}^2$ with $P(1,x,S)$.
\item $\epsilon > 0$ is small enough that $\frac{(2d+1)\epsilon}{2}< x$ (This $\epsilon$ will be fixed throughout the proof, and the $\epsilon$ in the definition of $\hookrightarrow$ will not factor into the proof).
\item $\delta_1, \delta_2,\delta_3,..., \delta_n$ are very small real numbers satisfying $\delta_2,\delta_3,...,\delta_n < \delta_1 << \epsilon$ and such that 
\begin{equation}\label{oneone} \frac{\epsilon - \delta_i}{\epsilon-\delta_j} \quad \text{ for } 1\leq i\neq j \leq n, \end{equation}
 are all irrational.
\end{itemize}

\subsection{Acknowledgements} I am very grateful to Richard Hind and Dan Burns for their interest in this project and their continued support. I would also like to thank Dan Cristofaro-Gardiner for a helpful discussion on special cases in the proof above.

\section{The Symplectic Ellipsoid}\label{ell}
In this section we consider the Reeb dynamics on the boundary of a symplectic ellipsoid, which is an example of a toric domain, to be defined shortly. As the name implies, an ellipsoid is a higher-dimensional generalization of a two-dimensional ellipse. When the ratio of the ellipsoid capacities (e.g. semi-major axis to semi-minor axis) is irrational, then the Reeb dynamics on the boundary are easy to understand. Moreover, when the ellipsoid is sufficiently eccentric (a so-called ``skinny'' ellipsoid) then there will only be one Reeb orbit of any consequence. This is precisely the setup we shall leverage when producing a holomorphic curve in the sequel.

A \textit{toric domain} is the preimage of some region in the first orthant under the \textit{moment map}
\[ \mu\colon \C^n \to \R^n, \qquad \mu(z_1,z_2,...,z_n) = (\pi|z_1|^2, \pi|z_2|^2,..., \pi|z_n|^2).\]

We first consider a symplectic ellipsoid, which  has nice Reeb dynamics on the boundary. Notice that the six-dimensional ellipsoid $E(r_1,r_2,r_3)$, as defined in section \ref{intro}, can also be expressed as the toric domain for the tetrahedral region below the face $\mathcal{R}$ that is shown in figure \ref{fig:sfig1}. This ellipsoid inherits the standard symplectic form from $\C^3\cong \R^6$, setting $z_j = x_j+iy_j$, and this form restricts to a contact form $\alpha$ on $\partial E(r_1,r_2,r_3)$. Specifically
\[ \alpha = \frac{1}{2}\left( x_1dy_1 - y_1dx_1 + x_2dy_2 - y_2dx_2 + x_3dy_3 - y_3dx_3 \right).\]
This contact form generates a Reeb vector field
\[ v = \frac{2\pi}{r_1}\left(x_1\pd{}{y_1} - y_1\pd{}{x_1} \right) + \frac{2\pi}{r_2}\left(x_2\pd{}{y_2} - y_2\pd{}{x_2} \right) +  \frac{2\pi}{r_3}\left(x_3\pd{}{y_3} - y_3\pd{}{x_3} \right) ,\]
along with a contact structure $\xi$ defined by $\xi = ker(\alpha)$.

 For the remainder of section \ref{ell}, assume that the triple $(r_1,r_2,r_3)$ satisfies the following condition
 \begin{equation}\label{irratcond} \frac{r_i}{r_j} \in \R\setminus\mathbb{Q} \quad \text{ for all } 1 \leq i \neq j \leq 3. \end{equation} 
  Then the boundary $\partial E(r_1,r_2,r_3)$ corresponds to the constant energy level set of a Hamiltonian mechanical system (i.e. a sum of three independent oscillators) that is completely integrable. The components of the moment map $\mu$, above, obviously Poisson commute, and fibers of this map are either $T^3$ or $T^2$ or $T^1$, as will be described shortly. It follows that there exist linear (action angle) coordinates on each torus fiber such that the Hamiltonian flow is straight-line motion. Choose a Hamiltonian (involving all coordinates of $\mu$) that generates the Reeb flow. Then we examine the three options for the fibers of the moment map $\mu$.
\begin{itemize}
\item A point $(z_1,z_2,z_3)$ within the open face of $\mathcal{R}$ will have no coordinates vanishing, and a fiber of the moment map over such a point will be $T^3$. A Reeb orbit in such a fiber cannot be closed, because the assumption (\ref{irratcond}) implies that $T\cdot (r_1,r_2,r_3) \in \R^3 \setminus \mathbb{Q}^3$ for every real $T>0$.
\item A fiber over the three edges of $\mathcal{R}$ (but not over a vertex) will be $T^2$. A Reeb orbit in such a fiber cannot be closed, because the assumption (\ref{irratcond}) implies that the slope is irrational. 
\item Finally, a Reeb orbit in a fiber over any vertex of $\mathcal{R}$ must be closed, because such a fiber is just $T^1=S^1$.\end{itemize}
 
 The conclusion is that the irrationality condition on the capacities precludes closed Reeb orbits on $\partial E$, save for the three orbits $\gamma_1$, $\gamma_2$, $\gamma_3$ that occur in the planes $\{ |z_i| = |z_j| = 0, \, i\neq j\}$.  Furthermore, the condition (\ref{irratcond}) applied to the capacities $r_1,r_2$ would be sufficient to preclude closed Reeb orbits if the ellipsoid were four dimensional, $E(r_1,r_2)$, a case we shall revisit later. For the remainder of this paper, the phrase ``Reeb orbit'' will always mean \emph{closed} Reeb orbit.

\begin{figure}
\begin{subfigure}{.5\textwidth}
  \centering
  \includegraphics[scale=1.5]{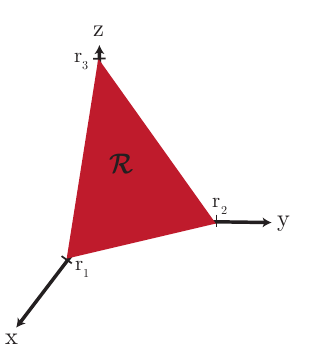}
  \caption{The region that generates\\ the ellipsoid $E(r_1,r_2,r_3)$.}
  \label{fig:sfig1}
\end{subfigure}%
\begin{subfigure}{.5\textwidth}
  \centering
  \includegraphics[scale=1.64]{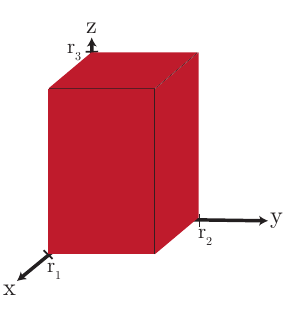}
  \caption{The region that generates the polydisc $P(r_1,r_2,r_3)$. The edges of the cuboid have been exaggerated here.}
  \label{fig:sfig2}
\end{subfigure}
\caption{The (solid) regions in the first octant that generate two important toric domains.}
\label{fig:fig1}
\end{figure}

\begin{ex}In the six-dimensional case, the convention $r_1 \leq r_2 \leq r_3$ means that we may name $\gamma =\gamma_1$ the \textit{short orbit}.  
 In coordinates $(x_1,y_1,x_2,y_2,x_3,y_3)$ as above, the Reeb orbits are parameterized as
 \[ \gamma_1(t) = \left( \frac{r_1}{\pi}\cos\left( \frac{2\pi t}{r_1}\right) , \frac{r_1}{\pi}\sin\left( \frac{2\pi t}{r_1}\right),0,0,0,0 \right) \qquad \text{ with period } r_1,\]
  \[ \gamma_2(t) = \left(0,0, \frac{r_2}{\pi}\cos\left( \frac{2\pi t}{r_2}\right) , \frac{r_2}{\pi}\sin\left( \frac{2\pi t}{r_2}\right),0,0 \right) \qquad \text{ with period } r_2,\]
   \[ \gamma_3(t) = \left(0,0,0,0, \frac{r_3}{\pi}\cos\left( \frac{2\pi t}{r_3}\right) , \frac{r_3}{\pi}\sin\left( \frac{2\pi t}{r_3}\right) \right) \qquad \text{ with period } r_3.\]
 The time $t$ Reeb flow $\text{Flow}_t\colon \partial E \to \partial E$ is easily seen to be the linear map
 \[ \text{Flow}_t(x_1,...,y_3) = \begin{bmatrix}Rot(2\pi t/r_1) & 0 & 0\\
 0 & Rot(2\pi t/r_2) & 0\\
 0& 0 & Rot(2\pi t/r_3) \end{bmatrix} \cdot \begin{bmatrix} x_1\\ \vdots\\ y_3\end{bmatrix}, \]
 where $Rot(\varphi)$ denotes the real $2\times 2$ matrix of rotation through angle $\varphi$, 
 \[ Rot(\varphi) = \begin{bmatrix} cos(\varphi) & -sin(\varphi) \\ sin(\varphi) & cos(\varphi) \end{bmatrix}.\]
 \end{ex}
 
 Along any Reeb orbit (the short orbit $\gamma = \gamma_1$, in particular) the Reeb flow gives a \textit{linearized return map} $P_{\gamma}\colon T_{\gamma(0)} \C^3 \to T_{\gamma(T)}\C^3$, which is defined to be the derivative of the time $T$ (\text{the period}) flow. The trivialization of these tangent spaces come from the inclusion into $\C^3$ (with the linearization acting trivially on the normal vector to $\partial E|_\gamma$).
  \begin{definition} The Reeb orbit $\gamma$ is said to be \emph{nondegenerate} whenever $P_\gamma$ does not have $1$ as an eigenvalue. Furthermore, $\gamma$ is called \emph{elliptic} if the eigenvalues of $P_{\gamma}$ lie within $S^1$, or $\gamma$ is called \emph{hyperbolic} if the eigenvalues of $P_{\gamma}$ are real.
 \end{definition}

\subsection{Generalized Conley-Zehnder Indices}\label{generalCZ}

We define an index formula for a continuous path of symplectic matrices, which may have $1$ as an eigenvalue. This is known as the generalized Conley-Zehnder Index, or the  Robbin-Salamon index of the path of matrices. This definition is taken from \cite{Gutt}, suitably modified.

\begin{definition} If $\gamma$ is a closed (possibly degenerate) Reeb orbit with period $T$, we might view $\gamma$ as a map having domain $\R/\,T$, a circle. Let $\tau$ denote a trivialization of the contact structure of the image space along $\gamma$. With respect to the trivialization $\tau$, the linearized Reeb flow is a symplectic matrix, denoted $P_\gamma(s)$ for $s \in \R/\,T$. Specifically, we regard $P_\gamma(s) \colon (\xi_{\gamma(0)},\omega_{std}) \to (\xi_{\gamma(s)},\omega_{std})$ as a $(2n-2)\times (2n-2)$ matrix. To the path of such matrices $\{ P_\gamma(s) \, \vert \, 0 \leq s \leq T \},$ we associate a \emph{generalized Conley-Zehnder index}, denoted $CZ(\gamma)$. The Conley-Zehnder index is computed by regarding the graph of $P_\gamma(s)$ as a path of Lagrangian subspaces of $(\R^{2n-2}\times \R^{2n-2}, (-\omega_{std})\times \omega_{std})$ and then taking the Robbin-Salamon index of this graph relative to the diagonal $\Delta = \{ (x,x)\in \mathbb{R}^{2n-2} \times \mathbb{R}^{2n-2}$\}. \end{definition}

Note that the generalized Conley-Zehnder index is half-integer valued, but it coincides with the traditional Conley-Zehnder index in the case when the Reeb orbit $\gamma$ is non-degenerate, a fact we use now. 
The following two claims are immediate consequences of this definition. Recall that we are using the trivialization coming from the inclusion of an ellipsoid into $\C^n$. (See, for example, \cite{Ustilovsky} p.14.)

\begin{claim}
Suppose $A \colon [0,T] \to \mathbb{R}^{2n_1}$ and $B\colon [0,T] \to \mathbb{R}^{2n_2}$ are two continuous paths of symplectic transformations. Identify $Sp(2n_1)\times Sp(2n_2)$ as the subgroup of block-diagonal matrices within $Sp(2n_1+2n_2)$. 
Then the Conley-Zehnder index enjoys the following direct sum identity.
\[ CZ\left(t \mapsto A\oplus B \right) = CZ\left( A \right) + CZ\left(B \right).\]
\end{claim}

\begin{claim}
A rotation in the plane by angle $2\pi t$ is the transformation $Rot(2\pi t)$ defined above. The path $t \mapsto Rot(2\pi t)$  for $0 \leq t \leq T$ has Conley-Zehnder index
\begin{align*}
 CZ(Rot(2\pi t))
 & =  \lfloor T \rfloor + \lceil T \rceil = 2\lfloor T \rfloor+1.
\end{align*}
\end{claim}

\begin{definition} The total rotation angle, $T$, of such an elliptic Reeb orbit is called the \emph{monodromy angle}.
\end{definition}

Putting these above two claims together, we get the following important result about the short orbit $\gamma = \gamma_1$.


\begin{lemma}\label{czgam} For the ellipsoid $E(r_1,r_2,r_3)$, the one parameter group of diffeomorphisms given by the time $t$ flow along $\gamma$ induces a symplectic linear map $\phi_t \colon \xi_{\gamma(0)} \to \xi_{\gamma(t)}$. Note that $r_1$ is the period of this short obit. For $0 \leq t \leq r_1$, the path $t \mapsto \phi_t$ has Conley-Zehnder index
\[ CZ(\gamma) = 4+ 2\left\lfloor \frac{r_1}{r_2}\right\rfloor + 2\left\lfloor \frac{r_1}{r_3}\right\rfloor.\]
More generally, consider $E(r_1,...,r_n)$ with all $r_i/r_j\in\R\setminus\Q$. Let $\gamma^{(r)}$ denote the $r$-fold cover of the short orbit, which is the path $\gamma(t)$ for $0 \leq t \leq r\cdot r_1$. The corresponding path of symplectic transformations has Conley-Zehnder index
\[ CZ(\gamma^{(r)}) = 2r + 2\left\lfloor \frac{r\cdot r_1}{r_2} \right\rfloor +...+ 2\left\lfloor \frac{r\cdot r_1}{r_n} \right\rfloor + n-1. \]
\begin{proof}
The second claim from Ustilovsky \cite{Ustilovsky} above implies that re-scaling the rotation $Rot(2\pi t)$ as $Rot(2\pi r t/r_2)$ should re-scale the Conley-Zehnder index as $2\lfloor rT/r_2\rfloor + 1$, and similarly for rotations in the other coordinate planes. Specifically, in the first coordinate plane we get $2\lfloor rr_1/r_1\rfloor=2r$, because the period of the orbit we consider is $T=r_1$. In the first coordinate plane there is no term $+1$, because the time one map is the identity. Using the first claim of Ustilovsky and the above expression for $\text{Flow}_t$, for $0\leq t \leq r_1$, we have
\[ CZ\left(\gamma^{(r)}\right) = 2r + \left[  \left(2\left\lfloor \frac{r\cdot r_1}{r_2}\right\rfloor+1\right) +...+ \left(2\left\lfloor \frac{r\cdot r_1}{r_n}\right\rfloor + 1\right)\right].\]
The terms in square brackets comprise a sum of $(n-1)$ summands, which are identical except for the denominator inside the floor function. The desired formula follows by gathering terms. 
\end{proof}

\end{lemma}

 Finally, this short orbit will have action $\pi r_1$. The monodromy angle scales as $r_1/r_2$ for $E(r_1,r_2)$.

\section{A Lagrangian Torus}\label{polyd}
Recall that we are assuming that $P(1,x)\times \R^2 \hookrightarrow P(a,b)\times \R^2$. In this section, we regard the symplectic polydiscs under consideration as toric domains, and we compactify the target polydisc. A neighborhood of a Lagrangian torus in the source polydisc will exhibit interesting Reeb dynamics (as did the ellipsoid in the previous section). We shall write down a Conley-Zehnder formula for these Reeb orbits in this section. Finally, we will consider pseudoholomorphic curves in (a compactification of) $P(a,b)\times \R^2$.

Now, let $S>>da+b$. As mentioned in section \ref{intro}, because of the definition of $\hookrightarrow$, we replace $P(1,x)\times \mathbb{R}^2$ with $P(1,x,S)$. The symplectic polydisc $P(1,x,S)$ is a toric domain for the cuboid in the first octant having side lengths $1,x$, and $S$. We define $U$ to be the subset of $P(1,x,S)$ defined by
\[ U = \mu^{-1}(\, (1-\epsilon,1)\times(x-(2d+1)\epsilon,x)\times(S/2,S) \,).\]

In symplectic polar coordinates, one can view a six-dimensional toric domain as toric fibers over a base given by a region in the first octant, as above. Let us view the symplectic polydisc $P(1,x,S)$ in symplectic polar coordinates:
\[ R_i = \pi|z_i|^2 \qquad \qquad \theta_i = arctan(y_i/x_i)\]
for $z_j = x_j +iy_j$ in the identification $\C^n \cong \R^{2n}$.
Then there is a Lagrangian torus
\[ L = \left \{ \pi|z_1|^2 = 1-\frac{\epsilon}{2}, \, \pi|z_2|^2 = x - \frac{(2d+1)\epsilon}{2}, \, \pi|z_3|^2 = \frac{3S}{4} \right\}\]
inside the set $U$. 

Using the moment description above, we can give a complete description of the Reeb orbits on $\partial U$, similar to the characterization given at the beginning of section \ref{ell}. Under the moment map, $U$ maps to a parallelepiped (cuboid) $P$ away from the coordinate axes of the first octant of $\R^3$. The Reeb flow on $\partial U$ is always tangent to the Lagrangian torus fibers of the moment map, and the direction is given by the  normal vector to $P$. Hence, we find periodic Reeb orbits (appearing in a $2$-dimensional family which will foliate the fiber torus) precisely when the normal vector to $P$ is a multiple of a nonzero integer vector.

\begin{remark}\label{mark} 
Unfortunately, in this moment description, the unit tangent bundle will not be smooth. This can be remedied by defining $\Sigma$ to be a smoothing of $\partial U$ obtained in a special way, described in the next paragraph.
Note that $U$ admits a symplectic embedding into $T^*T^3$ which takes $L$ to the zero section.
We use this identification to regard $U$ as a tubular neighborhood of the zero section in the cotangent bundle of the Lagrangian torus $L$, defined by a Finsler metric. Then the Reeb orbits on $\partial U$ correspond to geodesics on the base. More precisely, if $c(t)$ is a geodesic, then $(c, dc^*)$ will be a Reeb orbit in $\partial U$, and vice-versa. Such closed geodesics are parameterized by $H_1(T^3,\mathbb{Z})\setminus (0,0,0)$. In this way we see that closed Reeb orbits on $\partial U$ are in bijective correspondence with some triples of integers. Our distinguished torus has dimension $3$, which means that geodesics on the torus appear in $2$-dimensional families (given by translating the geodesics). Using the described correspondence in general, Reeb orbits in $\partial U$ must  appear in $(n-1)$-dimensional families. \end{remark}

Define a Hamiltonian on the cotangent bundle of the Lagrangian torus $L$, with fiber coordinates $y_1,y_2,y_3$ by the formula
\[ H = \left(\frac{|y_1|^p}{\epsilon^p} + \frac{|y_2|^p}{((2d+1)\epsilon)^p} + \frac{2|y_3|^p}{S^p}\right)^{1/p}, \]
where $p$ is some large integer. Then the set $H^{-1}(1) \subset T^*T^3$, centered at the Lagrangian torus $L$, will be a smoothing of $\partial U$.  We name this smoothing $\Sigma$ and henceforth consider only Reeb orbits on $\Sigma$. Note that for $p$ large, $\Sigma$ is close to $\partial U$ in the Hausdorff topology, and the Reeb orbits on $\Sigma$ approximate the orbits on $\partial U$.

The Reeb orbits along $\Sigma$ are pseudo-hyperbolic, and the formulas in the previous section will not suffice to compute the Conley-Zehnder index.
\begin{lemma}\label{CZP}
Let $c\colon [0,1] \to \Sigma$ be a closed Reeb orbit on $\Sigma$ in homology class $(k,\ell,m) \in H_1(T^3)\setminus (0,0,0)$. With respect to the trivialization coming from the inclusion into $\C^3$, the generalized Conley-Zehnder index of such an orbit is
\[ CZ(c) = 2k+2\ell+2m +1.\]
\end{lemma}
\begin{proof} The so-called symplectic shear matrix $A(t) = \begin{bmatrix}  1 & -t/2 \\ 0 & 1\end{bmatrix}\/$, for $t\in[0,1]$, has generalized Conley-Zehnder index $CZ(A(t)) = \frac{1}{2}$. With respect to the given trivialization, the orbit $c(t)$ has linearized Reeb flow given as a block-diagonal
\[ c(t) = \begin{bmatrix} e^{it} && \\ & A(t)  &\\ && A(t) \end{bmatrix}, \qquad t\in [0,1]. \] 
Now, using the product and shear properties of Lemma 26 of \cite{Gutt}, we find that
\[ CZ(c) = 2(k+m+\ell) + \frac{1}{2} + \frac{1}{2}, \]
as desired. Note that because the Reeb orbits along $\Sigma$ come in Morse-Bott families, we needed the added generality of the definition given in section \ref{generalCZ}.

\end{proof}

Now, choose $\delta_2,\delta_3<\delta_1 << \epsilon$ very small, and so that the condition (\ref{oneone}) in the list of section \ref{intro} holds. Define 
\[ E = E(\epsilon - \delta_1, (2d+1)(\epsilon-\delta_2), (2d+1)(\epsilon-\delta_3)).\]
Then $E$ symplectically embeds into $U$, and this embedding clearly extends to the closure of the ellipsoid. So we will not denote it by a hooked arrow. Furthermore
\[ X = \CP{1}(a)\times \CP{1}(b)\times \C \]
is a compactification of the first factor of $P(a,b)\times \R^2$ endowed with symplectic form 
\begin{equation}\label{bigo} \omega_X = a^2\omega_{FS}+b^2\omega_{FS}+\omega_{std}. \end{equation}
Here $\omega_{FS}$ is the Fubini-Study form on $\CP{1}$ and $\omega_{std}$ is the standard symplectic form on $\C \cong \R^2$, which was written in the introduction. We now have
\[ E \rightarrow U \hookrightarrow P(a,b)\times \R^2 \hookrightarrow X.\]

The reader may wonder as to why the $\C$ factor remains in $X$. In future sections of this paper, we will consider holomorphic curves in $X$, whose projection onto the $\C$ factor will always be bounded. More precisely, we shall fix the homology class of such holomorphic curves, which will fix the area and constrain the diameter. In short, there is no reason to worry about compactifying the $\C$ factor.

\subsection{$J$-holomorphic curves}
We conclude this section by describing the almost-complex geometry of the embedded $E \rightarrow U \hookrightarrow X$.

 \begin{definition} A symplectic manifold $(M, \omega)$ is said to have (one or more) \emph{cylindrical end(s)} if there is a compact codimension $0$ submanifold $K$ together with a symplectomorphism $M\setminus K \cong \partial K \times (0,\infty)$. Further, we require that $M$ be equipped with an almost-complex structure $J$ which is translation-invariant on $M\setminus K \cong \partial K \times (0,\infty)$. 
 \end{definition}

\begin{definition} Let $(M,\omega, J)$ be any symplectic manifold with compatible almost-complex structure. A \emph{$J$-holomorphic curve} is a map $u \colon (S^2 \setminus \Gamma,j) \to (M,J)$ satisfying the following PDE:
\[ du\circ j = J \circ du.\]
 Here  $\Gamma$ is a finite collection of punctures and $j$ is a complex structure that makes $S^2\setminus \Gamma$ into a Riemann surface. Finally, $u$ is asymptotic to closed Reeb orbits on $\partial M$ at each of the punctures. 
\end{definition}

\begin{definition} A \emph{$J$-holomorphic plane} is a $J$-holomorphic curve whose domain is a once-punctured sphere. To fix notation, we will henceforth consider the domain of a $J$-holomorphic plane to be $\C$ with the standard complex structure. 
\end{definition}

In $\CP{1}$, let $p_\infty$ denote the point at infinity. Abbreviate $\CP{1}\times \CP{1}\times \CP{1}$ as $(\mathbb{P}^1)^3$. Then our ambient manifold $X$ can be realized as $(\mathbb{P}^1)^3 \setminus (\CP{1}\times\CP{1}\times p_\infty)$. Now, inside of $(\mathbb{P}^1)^3$ there are two other important divisors, which are 
\[ D_1 = p_\infty \times \CP{1} \times \CP{1} \qquad \text{ and } \qquad D_2 =\CP{1} \times p_\infty \times \CP{1}.\] 
Let $[D_1], [D_2] \in H^2((\mathbb{P}^1)^3)$ denote the classes given by $c_1(L(D_1))$ and $c_1(L(D_2))$ (where $L(D_j)$ denotes the complex line bundle over the divisor $D_j$). Restrict each of $[D_1]$ and $[D_2]$ to $(\mathbb{P}^1)^3 \setminus (\CP{1}\times\CP{1}\times p_\infty)$ to obtain two non-degenerate, independent functionals on $H_2(X)$.

\begin{note} $[D_1]$ and $[D_2]$ are dual to the homology classes of the curves 
\[ C_1= \CP{1} \times p_\infty \times \{1\} \qquad \text{ and } \qquad C_2=p_\infty\times \CP{1}  \times \{1\}, \] 
and the classes $[C_1],[C_2]$ are an integral basis for $H_2(X)$. Notice that $[C_2]$ is homologically equivalent to $[\{1\}\times \CP{1}  \times \{1\}]$ which is disjoint from the divisor $p_\infty \times \CP{1} \times \CP{1}$. This shows that $[D_1]$ as so defined will pair trivially with $[C_2]$. In fact, after putting the basis curves $C_1, C_2$ into general position, we find that the intersection matrix $(a_{ij})=([D_i] \cdot [C_j])$ will be the $2\times 2$ identity matrix.
\end{note}

Heuristically, we can regard the pairing of an element of $H_2(X)$ as counting the intersection number of the homology class with the Poincar\'e dual of
\[ L_\infty = p_\infty \times \CP{1}(b)\times \C \cup \, \CP{1}(a)\times p_\infty \times \C \]
within $X$. It follows from \cite{sft} that we can find an almost-complex structure on $X\setminus E$ with cylindrical end on $\partial E$ and equal to the standard product structure near $L_\infty$. Notice that $L_\infty$ is a union of $4$-dimensional submanifolds of the $6$-manifold $X$. As described, there will be a well-defined intersection number obtained by pairing $L_\infty$ with a closed, oriented  $2$-dimensional submanifold. We will occasionally abuse notation and identify a $J$-holomorphic curve $u$ with its image in $X$. Notice that since we are in the $J$-holomorphic category, provided $u$ does not lie completely in $L_\infty$, then the intersection of $u$ with $L_\infty$ will consist of a finite number of points. We may assume $L_\infty$ remains a $J$-complex submanifold, under suitable deformations of $J$, so that each intersection with a (deformed) $J$-holomorphic curve $u$ will contribute a \emph{positive} intersection number.

\begin{definition} Let $u\colon \C \to X\setminus E$ be a $J$-holomorphic plane. Let ${C}$ denote the image of $u$ with the boundary orbit on $\partial E$. Express $[{C}]$ as an element of $H_2(X,E)$ with respect to the basis given by the restrictions of $[C_1], [C_2]$ in the note above. We say that the ordered pair of intersection numbers of $[{C}]$ with $[D_1]$ and $[D_2]$ is the \emph{bidegree}, $(d_1,d_2)$, of $u$. These intersection numbers are computed away from $E$. Alternatively, we can contract $E$ to a point and compute the intersection numbers in $H_2(X)$ of $C$ with contracted boundary orbits.
\end{definition}

 As above, let $\gamma = \gamma_1$ denote the short Reeb orbit on $\partial E$. Recall that a finite energy plane can have its end asymptotic to a \emph{cylindrical end} of a symplectic manifold.

 \begin{ex} The unit disc in $\C$ is a compact symplectic manifold. It admits a compatible almost-complex structure with a cylindrical end which makes it biholomorphic to all of $\C$, a non-compact manifold. This is the reason why we can regard finite-energy $J$-holomorphic planes as having domain $\C$ in the definition above. Recapitulating, we can either view a cylindrical end as the attachment of a half cylinder to a symplectic manifold,
or a manifold with cylindrical end can be obtained by fixing the almost-complex structure near the boundary of the symplectic manifold.
\end{ex}

Another example of a manifold with cylindrical end is a \emph{completion} of a compact symplectic manifold with contact boundary. We view the manifold $X\setminus E$ as a symplectic cobordism, and we associate a tame almost-complex structure to endow the boundary with a cylindrical end. (Some authors define a symplectic cobordism to automatically include completions.) The cylindrical end should be compatible with the Liouville contact structure on $\partial E$. This notion of completion will be explained further in sections \ref{neck} and \ref{curveproof}.

\begin{note}At several points during the course of this paper, we will restrict the class of almost-complex structures on $X\setminus E$, while keeping the symplectic form (\ref{bigo}) fixed. The restriction of the almost-complex structures will be very involved, because we shall perform two \textit{neck-stretching} operations: one in section \ref{neck} and one in section \ref{curveproof}. Nonetheless, the results of symplectic field theory, \cite{sft}, ensure that we have an ample supply of almost-complex structures to draw upon. Here is the current list of restrictions on such $J$.
\begin{definition}
Let $\mathcal{J}^\star$ denote the list of almost-complex structures on $(X\setminus E, \omega)$ which are
\begin{itemize}
\item compatible (and tame) with respect to the symplectic form $\omega$ on $X$ defined in (\ref{bigo}),
\item cylindrical at $\partial E$ and compatible with the Liouville contact structure on $\partial E$ (c.f. def. 4.2 of \cite{partitions}),
\item equal to a standard product structure outside of a compact set, making $L_\infty$ into a complex submanifold.
\end{itemize}
\end{definition}
Again, this list of conditions will be updated in future sections of the paper. \end{note}

\section{Stretching the Neck}\label{neck}
In this section we describe an important existence theorem of a $J$-holomorphic curve. This is the curve we shall consider when we stretch the neck along the smoothing $\Sigma$ of $\partial U$. After stretching the neck, the curve limits to a holomorphic building, and we examine the limiting curves that can occur in the building.

Recall that, by assumption, we have, a string of embeddings
\[ E(\epsilon - \delta_1, (2d+1)(\epsilon-\delta_2), (2d+1)(\epsilon-\delta_3))=E \rightarrow U \hookrightarrow P(a,b)\times \R^2 \hookrightarrow X.\]

We can now state the main existence theorem of this paper.
 
\begin{thm}\label{bigcurve}
In (a completion of) $X \setminus E$ there exist $J$-holomoprhic planes of bidegree $(d,1)$ asymptotic to $\gamma_1^{(2d+1)}$. Such curves persist under scaling of the ellipsoid and under a special class of deformations of the almost-complex structure, to be explained in the proof. 
\end{thm}

The proof of this theorem is long, and will require a technical procedure known as \textit{stretching the neck} along $\partial E$. In order to not interrupt the flow of the main argument, we have placed the proof of theorem \ref{bigcurve} in section \ref{curveproof}. The $J$-holomorphic curve that is produced in theorem \ref{bigcurve} will have area $da+b-(2d+1)(\epsilon-\delta_1)$.

Actually, the main argument will require another neck-stretching procedure, this time along a smoothing $\Sigma$ of $\partial U$, which will be described in this section. Note that the special class of deformations of the almost-complex structure that is mentioned in theorem \ref{bigcurve} will allow for this second stretching along $\Sigma$. The author hopes that the exposition in this section will make the stretching in section $\ref{curveproof}$ easier to understand. It is possible, however, to read section $\ref{curveproof}$ before this section. The current state-of-the-art on neck stretching is described in \cite{stretching}, but here we are in the simpler case of stretching along a submanifold of contact type. 

\begin{definition} We say $W$ is a \emph{contact-type hypersurface} in some symplectic manifold  $(M^{2n},\omega)$ if there is a neighborhood of $W$ on which the symplectic form $\omega$ is exact, say $\omega = d\lambda$, and the corresponding Liouville vector field $v$ defined by $\omega(v , \cdot) = \lambda (\cdot)$ is transverse (not tangent) to $W$. \end{definition}

\begin{lemma} In this context, the primitive
$\lambda|_{W}$ is a contact form on $W$. It induces a contact structure $\xi = \{ \lambda = 0 \}$, which is a $(2n-2)$-hyperplane distribution on $W$.
\end{lemma}

\begin{proof}
We must show that $\lambda \wedge(d\lambda)^{n-1} = \lambda \wedge \omega^{n-1}$ is a volume form on $W$. Now, $ker(\omega|_{W})$ is non-trivial, because $W$ is odd-dimensional. Further $ker(\omega|_{W}) \subset TW^{\perp \omega}$, which is $1$-dimensional. We therefore have that $dim(ker(\omega|_{W})) = 1$. Choose a non-zero vector field $Y_1 \in ker(\omega|_{W})$, and complete this to a frame $\{ Y_1, Y_2, ... , Y_{2n-1} \}$ for $TW$. We evaluate $\lambda \wedge \omega^{n-1}$ on this frame to get a single non-zero term
\[ \lambda|_{W}(Y_1)\omega|_{W}(Y_2,...,Y_{2n-1}) = \omega(v,Y_1)\omega(Y_2,...,Y_{2n-1}) \neq 0.\]\end{proof}

Now that we understand the setup, the contact-type hypersurface we consider is $W = \Sigma$, a smoothing of $\partial U$.
$\Sigma$ separates $X$ into two components, which by abuse of notation we call $U$ and $X\setminus U$, (i.e. when abusing notation we ignore the smoothing that defines $\Sigma$). Let $v$ denote the Liouville vector field for $\Sigma$. An almost complex structure $J$ on $X$ is said to be \textit{compatible} with the stretching if the contact structure $\xi =\{\lambda = 0 \}$ on $\Sigma$ coincides with
\[ T(\Sigma) \cap JT(\Sigma) \]
and if $J(v)$ equals the Reeb vector field of $\lambda$. We update the definition of $\mathcal{J}^*$ by requiring the almost-complex structures to be compatible with the stretching along $\Sigma$. 

Now, for every $N \in \mathbb{N}$, the three stretched manifolds 
\begin{equation}\label{parts} \overline{A_N}:=(U\setminus E,\, e^{-N}\omega), \qquad \overline{B_N}:=(\Sigma \times [-N,N], \,d(e^t\lambda)), \qquad \overline{C_N}:=(X \setminus U, \, e^N\omega) \end{equation}
can be glued along their boundary components to obtain a symplectic manifold, which we denote by $(\overline{X^N}, \omega^N)$. Here $t$ denotes the real coordinate on $[-N,N]$. Choose $J \in \mathcal{J}^*$, and  let $J^N$ be the continuous almost-complex structure on $\overline{X^N}$ that equals $J$ on the pieces $\overline{A_N}$ and $\overline{C_N}$ and is translation invariant on piece $\overline{B_N}$. Some perturbation may be necessary to smooth this choice of almost-complex structure, but this will not affect the results below. (See \cite{sft} for more details on this smoothing.) So we may as well assume that all $J^N$ are smooth.

After all this setup, we have that the Reeb orbits on $\Sigma$ are in $2$-dimensional families indexed by triples $(k,\ell,m)\in \mathbb{Z}^3\setminus {\vec{0}}$, as described in the remark on page \pageref{mark}. Specifically, $k$ counts the winding in the $(x_1,y_1)$ factor, and so on. The Reeb orbits along $\Sigma$ may appear in a homology class with $k,\ell, m < 0$  (as opposed to $\partial E$ where winding along the short orbit is described  by  a non-negative integer).
After stretching the neck, the limiting curve (as $N \to \infty$) is a holomorphic building, i.e. a collection of finite energy holomorphic curves in completions of $\overline{C_\infty}:=X\setminus U$ and $\overline{A_\infty}:=U \setminus E$ with matching asymptotic limits along $\Sigma$. We shall name these completions in the compactness theorem below, with less cumbersome notation.

\begin{thm}\label{building}(See 10.6 of \cite{sft}) Fix a $J$ in $\mathcal{J}^*$. For each $N \in \mathbb{N}$, let $u_N$ be a $J^N$-holomorphic curve negatively asymptotic to $\gamma_1^{(2d+1)}$. Fix a representative $f_N$ for each $u_N$. Then there exists a subsequence of the $f_N$ which converges to a holomorphic building $\textbf{F}$. The domain of $\textbf{F}$ is a nodal Riemann sphere $(S,j)$ with punctures, and the building can be described as a collection of finite energy holomorphic maps from the collection of punctured spheres $S \setminus \{\text{nodes}\}$ into one of the following three completions:
\[ A_\infty : = (U\setminus E) \cup_{\Sigma} (\Sigma \times [0,\infty)) \qquad \text{ with form } d(e^t\lambda),\, t\in [0,\infty),\]
\[ B_\infty : = B_\infty = \Sigma \times \R,\qquad  \text{ with form } d(e^t\lambda), \, t\in \R ,\]
or
\[ C_\infty : = (X \setminus U) \cup_{\Sigma} (\Sigma \times (-\infty,0]) \qquad \text{ with form } d(e^t\lambda),\, t\in (-\infty,0]. \]
\end{thm}

\label{facts} We now gather some facts about the completed manifold $C_\infty$. Notice that ${C_\infty}$ is diffeomorphic to $X\setminus U$, but ${C_\infty}$ has a specific almost-complex structure near the boundary. This is an example of a manifold with cylindrical end. Consequently $J$-holomorphic curves into ${C_\infty}$ can have only negative asymptotic ends. Each curve of $\textbf{F}$ having target in $C_\infty$ will be asymptotically cylindrical near the punctures in its domain, hence will wind around the Reeb orbits of $\Sigma$ as $t\to -\infty$. We see that this neck-stretching construction sends all removable singularities of a curve to the Reeb orbit ``at $-\infty$''.  As mentioned above, we can parameterize the Reeb orbits on $\Sigma$ by three integers. The area of a curve $u$ is defined using the area form coming from the completion of $X\setminus U$, not the re-scaled form on $C_N$ above.

\label{levelstr}The holomorphic building $\textbf{F}$ will have a \textit{level structure}, which we now explain. To encode a building $\textbf{F}$ of height $l$, we label all of the punctured Riemann surfaces $(S,j)$ by integers $0,1,...,l+1$. These labels are called levels. The difference between the levels of any two components of the building that share a node must be one. In this way we think of the levels of the building as heights, and we view curves in the building as glued along nodes to curves one level above or below. An illustration of such a building is given in figure \ref{fig:building}. Let $S_\kappa$ denote the union of components of level $\kappa$ and denote by $u_\kappa$ the holomorphic curve of \textbf{F} with domain $S_\kappa$. The domain $S_\kappa$ could be disconnected if bubbling occurs during the neck stretching, but we will not examine bubbling yet. In any case, we have that $u_0 \colon S_0 \to A_\infty$ and $u_\kappa \colon S_\kappa \to B_\infty$ for $1\leq \kappa \leq l$ and $u_{l+1} \colon S_{l+1} \to C_\infty$. Each node that is shared by $S_\kappa$ and $S_{\kappa+1}$ is a positive puncture for $u_\kappa$ and a negative puncture for $u_{\kappa+1}$, and each are asymptotic to the same Reeb orbit. A very important property of the definition from \cite{sft} is that we assume none of the curves $u_\kappa$ for $1\leq \kappa \leq l$ consist entirely of trivial cylinders over Reeb orbits, except in the following situation. If a collection of $J$-holomorphic planes converges to a limiting plane $v$, then after stretching the neck, this sequence of planes will converge (in the sense of SFT) to a building consisting of $v$ in the highest level together with a trivial cylinder in the symplectization over its limiting orbit. Since we are stretching the $J$-holomorphic plane that was produced in theorem \ref{bigcurve}, we are in this exceptional situation.
Notationally, we will not emphasize the number $l$ in the discussion that follows.

Note that part $B_\infty$ is a symplectization and parts $A_\infty$ and $C_\infty$ are symplectic cobordisms (with completions). To wit, a symplectic cobordism is a symplectic manifold $M$ with boundary $\partial M = M^- \sqcup M^+$ such that $M^+$ has an outward-pointing Liouville vector field and $M^-$ has an inward-pointing Liouville vector field. These Liouville fields define contact forms on $M^- \sqcup M^+$, which give rise to Reeb orbits on $\partial M$. It is important to note that the Reeb orbits on $\partial M$ come in Morse-Bott families. For this reason, we use the generalized Conley-Zehnder index, as defined in section \ref{generalCZ}. Furthermore, we will need to keep track of the dimensions of the families of the Reeb orbits that appear on $\partial M$.
We need to investigate what limiting curves arise in the holomorphic building that we get after stretching the neck along $\Sigma$. This investigation will require a virtual index formula for a $J$-holomorphic curve $u$ in a symplectic cobordism or a symplectization. There is an important quantity that we associate to a curve $u$, known as the \textit{relative first Chern class}. In parts $A_\infty$ and $B_\infty$ the relative Chern class will turn out to be zero. We examine $C_\infty$ here. Specifically, the symplectic cobordism under consideration is a null-cobordism ($\partial(X\setminus U)\cong \Sigma \sqcup \emptyset$) with completed end $\Sigma \times (-\infty,0]$. For the full definition of the relative first Chern class in a symplectic cobordism, see $\S4.2$ of \cite{partitions} (noting that $n=2$ there). The definition depends on a choice of trivialization $\tau$ of the contact structure along the Reeb orbits that are the ends of the $J$-holomorphic curve $u$. In our case, curves in $C_\infty$ can have only negative ends that wind along Reeb orbits of $\Sigma$. 
Here we choose a symplectic trivialization of $u^*TC_\infty$, which we deform near the punctures $\Gamma$ of $u$ so that it agrees with the trivialization on $\partial E$ or $\Sigma$ that we used to define the Conley-Zehnder index in section \ref{generalCZ} and lemma \ref{CZP} (the trivialization coming from inclusion into $\C^3$). Then we define $c_1(u)$ to be the algebraic count of of zeros of a generic section of $\Lambda^n(u^*TC_\infty)$ which is constant with respect to the trivialization on the boundary. The relative first Chern class of a curve $u$ in $C_\infty$  will be denoted $c_\tau(u)$. If $u$ has image $C$ in $C_\infty$ and has bidegree $(d_1,d_2)$, then with this choice of trivialization,
\[ c_\tau(u) : = c_1(\left.T(C_\infty)\right|_C,\tau )=2d_1 + 2d_2.\]
As with the bidegree, the computation of $c_\tau$ is designed to ignore the winding of the curve $u$ along its negative ends, where $\Lambda^n(u^*TC_\infty)$ is trivial.  Note that the Conley-Zehnder index and the relative first Chern class each depend on the choice of trivialization $\tau$, but the index formula that follows does not depend on the trivialization.

The general index formula for finite-energy genus zero curves asymptotic to Reeb orbits (taken from \cite{bougiethesis} Corollary 5.4) is 
\begin{equation*}\label{bougie} Index(u) = (n-3)(2- s^+ - s^-) + 2c_\tau(u) + \sum_{i=1}^{s^+} CZ(\gamma_i^+) +\frac{1}{2}\sum_{i=1}^{s^+} \dim(\gamma_i^+) - \sum_{j=1}^{s^-} CZ(\gamma_j^-) + \frac{1}{2}\sum_{j=1}^{s^-} \dim(\gamma_j^-), \end{equation*}
where 
\begin{itemize}
\item $n$ is half the dimension of the ambient space $X$ (usually $n=3$ in this paper),
\item $s^-$ is the number of negative ends,
\item $s^+$ is the number of positive ends (which will usually be zero in this paper),
\item $CZ(\gamma_j^-)$ is the generalized Conley-Zehnder index of the $j^{th}$ negative end, and
\item $\dim(\gamma_i^{-})$ counts the dimension of the family of Reeb orbits at each negative end of $u$.
\end{itemize}
This index formula predicts the dimension of the moduli space of $J$-holomorphic curves asymptotic to prescribed Reeb orbits. In this paper we shall call this Fredholm intex the \textit{virtual index}.
With the above assumptions (finite energy, genus zero, only negative ends), this formula simplifies to
\begin{equation}\label{generalindex} Index(u) = (n-3)(2-s^-)+  4d_1+4d_2 - \sum_{j=1}^{s^-} \left[CZ(\gamma_j^-) - \frac{1}{2}\dim(\gamma_j^-)   \right]. \end{equation}
More specific instances of this index formula will appear in the arguments that follow.

The virtual index of a \textit{closed} curve of genus $0$ is given by the formula
\begin{equation}\label{closedindex} Index(u) = (n-3)(2) + 2c_\tau(u).\end{equation}
The closed curve formula is a special case of the above, which can be found in $\S 2.2$ of \cite{ECHnotes}, among other sources. 

Genericity of $J$ is what allows the dimension of a moduli space of $J$-holomorphic curves to be predicted by the Fredholm index of a representative. To that end, we need the following lemma. 
\begin{lemma}\label{cinfpos}
For generic $J$, any $J$-holomorphic curve with image in piece $C_\infty$ will have non-negative virtual Fredholm index.
\end{lemma}
\begin{proof}
A $J$-holomorphic curve $u\colon S^2\setminus \Gamma \to C_\infty$ mapping into $C_\infty$ must have only  negative ends winding about some Reeb orbit in $\Sigma$.  If $u$ is somewhere injective, then it will have non-negative virtual Fredholm index, by genericity of $J$. Otherwise, $u$ multiply covers a somewhere injective underlying curve, which we call $\tilde{u}\colon S^2\setminus \tilde{\Gamma} \to C_\infty$. Let us say that $u$ is a $p$-fold covering of $\tilde{u}$. Then there is a ramified covering map $\psi \colon S^2\setminus\Gamma\to S^2\setminus\tilde{\Gamma}$ such that $u=\tilde{u}\circ \psi$. Since these singularities are removable, we can extend $\psi$ to a holomorphic map $\Psi\colon S^2 \to S^2$ sending $\Gamma \to \tilde{\Gamma}$. We write $s^-$ for the number of negative ends of $u$ and $\tilde{s}^-$ for the number of negative ends of $\tilde{u}$. Then $p\tilde{s}^- - s^-$ will be the total ramification of $\psi$ over $\Gamma$. By the Riemann-Hurwitz formula,
\begin{equation}\label{RieHurtz} p\tilde{s}^- - s^- = \sum_{x\in \Gamma}(m_x - 1) \leq \sum_{x\in S^2}(m_x-1) = 2p-2.\end{equation}

We compute the virtual Fredholm index of both $u$ and $\tilde{u}$, using the general index formula. 
As noted above, a negative end of either curve can be characterized by its homology class, so we write $(k_j, \ell_j, m_j)$ for the homology class of the $j^{th}$ negative end of $u$ and we write  $(\tilde{k}_j,\tilde{\ell}_j, \tilde{m}_j)$ for the $j^{th}$ negative end of $\tilde{u}$. The family of such orbits along $\Sigma$ has dimension $2$. Both curves have a bidegree, and we write $(d_1,d_2)$ for the bidegree of $u$, and we write $(\tilde{d}_1,\tilde{d}_2)$ for the bidegree of $\tilde{u}$. Now the index formula gives
\begin{equation}\label{indutild} Index(\tilde{u}) = (n-3)(2-\tilde{s}^-) + 4(\tilde{d}_1+\tilde{d}_2) - 2\sum_{j=1}^{\tilde{s}^-}\tilde{k}_j+\tilde{\ell}_j + \tilde{m}_j, \end{equation}
which is non-negative by assumption.
Similarly, we find
\begin{align*} Index(u)&=(n-3)(2-s^-) + 4(d_1+d_2) - 2\sum_{j=1}^{s^-} k_j+\ell_j + m_j\\
&=(3-n)(s^--2) + 4p(\tilde{d}_1+\tilde{d}_2) - 2p\sum_{j=1}^{\tilde{s}^-}\tilde{k}_j+\tilde{\ell}_j + \tilde{m}_j \\
&\geq (3-n)(p\tilde{s}^- - 2p) + 4p(\tilde{d}_1+\tilde{d}_2) - 2p\sum_{j=1}^{\tilde{s}^-}\tilde{k}_j+\tilde{\ell}_j + \tilde{m}_j \\
&= p\cdot index(\tilde{u})\geq 0
\end{align*}
using the inequality (\ref{RieHurtz}).
\end{proof}

\begin{lemma}\label{binfpos}
For generic $J$, any trivial cylinder with image in piece $B_\infty$ will have index $(n-1)$. 
\end{lemma}

\begin{proof}
Recall that, by our choice of trivialization, the relative first Chern class term of the index formula vanishes in  $B_\infty$.  A trivial cylinder is a simple curve in the symplectization layer $B_\infty$ which is topologically $\gamma \times \R$ for some Reeb orbit $\gamma$ on $\Sigma$. Hence $\gamma$ is both the positive asymptotic end and the negative asymptotic end of such a trivial cylinder. Obviously, the Conley-Zehnder terms in the index formula will cancel out for these two ends, being equal and opposite in sign. This implies that the index of such a trivial cylinder reduces to 
\begin{align*} Index = (n-3)(2-1-1) + (1/2+1/2)(n-1) =n-1 \end{align*}

\end{proof}

The following result tells us what limiting curves can appear in the holomorphic building $\textbf{F}$ that results from stretching the neck. We make use of the notion of matching curves into components in the  pseudoholomorphic building. A matched component is given by formally gluing together some collection of curve components lying in various levels of the building, with identifications being made along paired ends. We treat the domain of such a component as a smooth, connected, punctured Riemann surface. We define the positive ends of a matched component to be any positive ends of the constituent curves which have not already been matched to other curves in the sub-building. Similarly, we define the negative ends of a matched component to be any negative ends of constituent curves that have not already been matched. We are going to need to compute the virtual Fredholm index of these matched components, so we explain here how the Fredholm index behaves under the matching procedure. If a sub-building $B$ is obtained by matching curves $u_1,...,u_p$ along asymptotic orbits $\gamma_1,...,\gamma_q$ belonging to families of Reeb orbits in spaces $S_j$, then 
\begin{equation}\label{buildingind} Index(B) = \sum_{i=1}^p Index(u_i)  - \sum_{j=1}^q \dim(S_j). \end{equation}
In a similar way, if we were to form a sub-building $B$ by matching other sub-buildings $B_i$ along matching orbits $\gamma_i$ that belong to families of Reeb orbits in spaces $S_i$, then we have
\[ Index(B) = \sum Index(B_i) - \sum \dim(S_i).\]
Finally, after all these identifications are made, we can compute the index of a single sub-building $B$ by only considering its un-matched positive and negative ends and using the general index formula for a single curve.

\begin{thm}\label{happen} After stretching the neck along $\Sigma$, limiting curves can be matched into components consisting of
\begin{itemize}
\item some sub-buildings whose matching ends can be identified to form a plane with a negative end asymptotic to a Reeb orbits on $\Sigma$; and
\item one finite energy component with positive ends on $\Sigma$ and with a single negative end asymptotic to $\gamma^{2d+1}$ on $\partial E$. Let us call this one limiting component the \textit{special curve}.
\end{itemize}
Furthermore, the sub-buildings mentioned in this list all have non-negative building index.\end{thm}

\begin{proof} As mentioned above, the curves described in the two bullet points of the theorem are (potentially) several curves identified into connected components, according to the following identifications.


 There must be a unique curve in the lowest level of the holomorphic building, because there is only one negative end of the pre-stretched curve, asymptotic to $\gamma_1$ by construction. Let us say that the limiting curve in this lowest level is named $u_0$. Next, we shall identify with $u_0$ some curves in the symplectization layer, based on matching conditions, and we shall consider these identified curves to be a single component. 
 
 Beginning with $u_0$, locate \textit{all} curves that can be connected to $u_0$ through a chain of curves with matching ends in $B_\infty$.  Identify these curves along their matchings ends, and call the resulting component $v_0$. Because the negative end of $v_0$ is uniquely specified by $u_0$, we know that $v_0$ can have only positive ends that remain un-matched. Moreover, since the identification glues all matching ends in $B_{\infty}$, we see that the positive end(s) of $v_0$ must be asymptotic to Reeb orbits on $\Sigma$, the positive end of piece $B_\infty$. Consider the complement of $v_0$ in the holomorphic building. Let us say that, after identifying matching ends, this complement has exactly $M$ connected components. Name these components $v_1, ... , v_M$. Furthermore, we assume that $v_0$ is matched as far as possible with curves in $B_\infty$, so that the components $v_j$ for $j\in \{ 1,..., M\}$ do not contain any curves that could be matched with $v_0$. 
 
The curves $v_1,...,v_M$, described here, exist in the highest level(s) of the holomorphic building $\textbf{F}$ from theorem \ref{building}, and the negative ends of these curves must be asymptotic to a Reeb orbit on $\Sigma$. Notice that in the gluing construction of the previous paragraph, we are identifying curves in the manifold $B_\infty$ with curves in the manifold $A_\infty$ or $C_\infty$ even though these manifolds are technically disjoint. We make this identification, nonetheless, because the positive cylindrical end of $A_\infty$ (as $t\to \infty$) is equal to the negative cylindrical end of $B_\infty$ (as $t\to -\infty$). Moreover, if we glue the positive ends of $v_0$ with the negative ends of all the curves $v_1,...,v_M$, then we recover the original curve that was constructed in theorem \ref{bigcurve}. This original curve had genus zero, which forces the identified curve $v_0 \cup ... \cup v_M$ to have genus zero. Consequently, each of $v_0,...,v_M$ must be genus zero as a matched component.

In proposition \ref{prop5} of section \ref{curveproof}, we explain why a curve in the symplectization layer $B_\infty$ must be a trivial cylinder. Moreover, there can exist no $J$-holomorpic planes in part $A_\infty$, because the boundary of any such plane would have positive end asymptotic to $\Sigma$, and there are no contractible Reeb orbits on $\Sigma$. These two facts imply that, after the identifications above, the curves $v_1,...,v_M$ are not allowed to ``turn around'', re-enter the symplectization layer, and terminate in $A_\infty$. This fact helps the reader to visualize the curves in $C_\infty$. The key point is that no matter how many times a curve $v_j \in \{v_1,...,v_M\}$ enters or exits the symplectization layer, it won't affect the Fredholm index. Specifically, we proved in lemma \ref{cinfpos} that the ends of the components $v_1,...,v_M$ that lie in $C_\infty$ all have non-negative index. These ends may be identified with some trivial cylinders in part $B_\infty$ to form the sub-buildings $v_1,...,v_M$. We showed in lemma \ref{binfpos} that these trivial cylinders have index $(n-1)$, which is exactly the dimension of the family of Reeb orbits on $\Sigma$, where the matching takes place. Hence, this added index is exactly canceled by the dimension of the matching term in formula (\ref{buildingind}). In effect, we need only consider only the negative end and the bidegree of each of $v_1,...,v_M$ when computing its Fredholm (building) index, and we can ignore the contribution from trivial cylinders. From this analysis we conclude that each of $v_1,...,v_M$ have non-negative building index.

The curves $v_1,...,v_M$ have disjoint domains, by construction. In addition, the $v_j$ have finite energy. This discussion leaves, topologically, only one possibility for the curves $v_1,...,v_M$. Each must be a $J$-holomorphic plane whose negative end matches with one of the positive ends of $v_0$. This also leaves only one possibility for $v_0$. Topologically, $v_0$ is a pair of pants with $M$ pant legs. Figure \ref{fig:building} gives an illustration of these curves.


 The curves $v_1,...,v_M$ within $C_\infty$ must collectively have negative ends that bound a cycle in $U$ (since they are the boundary of $v_0$ plus a disc in $E$). If we denote the homology class of the negative end of the curve $v_j \in \{v_1,...,v_M\}$ by the ordered triple $(k_j,\ell_j,m_j)$, as described in lemma \ref{CZP}, then we must have
 \begin{equation}\label{sumk}
 \sum_{j=1}^M k_j = 0,
 \end{equation}
 and
 \begin{equation}\label{suml}
 \sum_{j=1}^M \ell_j = 0,
 \end{equation}
 and similarly for the $m_j$. The negative ends of $v_1,...,v_M$ must match up with all of the positive ends of $v_0$. The preceding calculation gives
 \begin{equation}\label{sumall} \sum_{j=1}^M k_j+\ell_j + m_j = 0,\end{equation}
 where here the sum is taken over all positive ends of $v_0$. The extra terms in the Conley-Zehnder index formula for any positive end of $v_0$, coming from the symplectic shears in lemma \ref{CZP} will be $\frac{n-1}{2}$. The dimension of the family of Reeb orbits at any positive end of $v_0$ is $(n-1)$. Once again, the special curve has one negative end asymptotic to $\gamma_1^{(2d+1)}$ and $M$ positive ends on $\Sigma$. The first Chern class vanishes for all constituent curves of $v_0$, hence also for $v_0$. For the same reason as above, we can ignore the index contribution of any trivial cylinders that were identified to $u_0$ to form $v_0$. In all, the general index formula gives
 \begin{align*}
 Index(v_0) =Index(u_0)&= (n-3)(2-M-1) + 0 + \sum_{j=1}^M CZ(\gamma_j^+)+\frac{1}{2}\sum_{j=1}^M (n-1) - [2(2d+1) +n-1] +0\\
 &= (n-3)(1-M) -2(2d+1) -(n-1) + \sum_{j=1}^M \left(k_j+\ell_j+m_j + \frac{n-1}{2}\right) + \frac{M}{2}(n-1)\\
&= (n-3) - (n-1) -M(n-3) +M(n-1) -2(2d+1)\\
& = 2(M-1) - 2(2d+1)
 \end{align*}
after incorporating formula (\ref{sumall}). We see that $v_0$ will have non-negative Fredholm index precisely when $M\geq 2d+2$, which is true by positivity of area of $v_0$. In a future proof (of theorem \ref{seeinfinity}) we shall explicitly compute that $v_0$ is of index zero. 
\end{proof}

\begin{figure}[h]
\begin{tikzpicture}
\tikzstyle{block} = [rectangle, draw,    text width=30em, text centered, rounded corners, minimum height=21em]
 
  \draw [-] (-1.38,0.05) to [bend right = 150] (1.38,0.05);
  \draw [-] (-1.38,-0.2) to [dashed, bend left = 150] (1.38,-0.2);
  \draw [-] (-1.5,-0.11) to [out =100,in=2750] (-4.89,5);
  \draw [-] (-3.11,5) [bend left = 1000, looseness=4] to (-2.75,5);
  \draw [-] (-1.25,5) to [out=260,in=270,looseness=5] (-0.39,5);
  \draw [-] (1.39,5) to [out=260,in=284,looseness=4.5] (1.99,4.98);
  \draw [-] (3.52,5) to [out=260,in=284,looseness=5] (4.51,5);
  \draw [-] (1.5,-0.11) to [out =70,in=270,looseness=2] (6.49,5);
  \draw (-4,5) ellipse (0.89 and 0.35);
  \draw (-2,5) ellipse (0.75 and 0.35);
  \draw (0.5,5) ellipse (0.89 and 0.35);
  \draw (2.75,5) ellipse (0.77 and 0.35);
  \draw (5.5,5) ellipse (0.99 and 0.35);
  \draw (-4,5.8) ellipse (0.89 and 0.35);
  \draw plot [smooth] coordinates {(-4.89,5.8)(-4.4,8)(-4.5,10)(-4,11)(-3.6,10)(-3.5,8)(-3.11,5.8)};
  \draw (-2,5.8) ellipse (0.75 and 0.35);
  \draw [-] (-2.75,5.8) [bend right = 1000, looseness=4] to (-1.25,5.8);
  \draw (0.5,5.8) ellipse (0.89 and 0.35);
  \draw [-] (-0.39,5.8) to [out=100,in=95,looseness=4.5] (1.39,5.8);
  \draw (2.75,5.8) ellipse (0.77 and 0.35);
  \draw [-] (1.98,5.8) to [out=85,in=100,looseness=4.5] (3.52,5.8);
  \draw (5.5,5.8) ellipse (0.99 and 0.35);
  \draw [-] (4.51,5.8) to [out=100,in=80,looseness=4.5] (6.49,5.8);
  
  \node at (4.05,4.5){$\hdots$};
  \node at (4.05,6.5){$\hdots$};
  \node at (0,2){$v_0$};
  \node at (-4,7){$v_1$};
  \node at (-2,7){$v_2$};
  \node at (0.5,7){$v_3$};
  \node at (2.75,7){$v_4$};
  \node at (5.5,7){$v_M$};
  \node at (7.75,2.3){\fontsize{150}{70}\selectfont ]};
  \node at (9.25,2.3){in $A_\infty$ and $B_\infty$};
  \node at (9.25,8){in $C_\infty$ and $B_\infty$};
  \node at (7.75,8){\fontsize{150}{70}\selectfont ]};
\end{tikzpicture}
\caption{The curves $v_0,...,v_M$.}
\label{fig:building}
\end{figure}
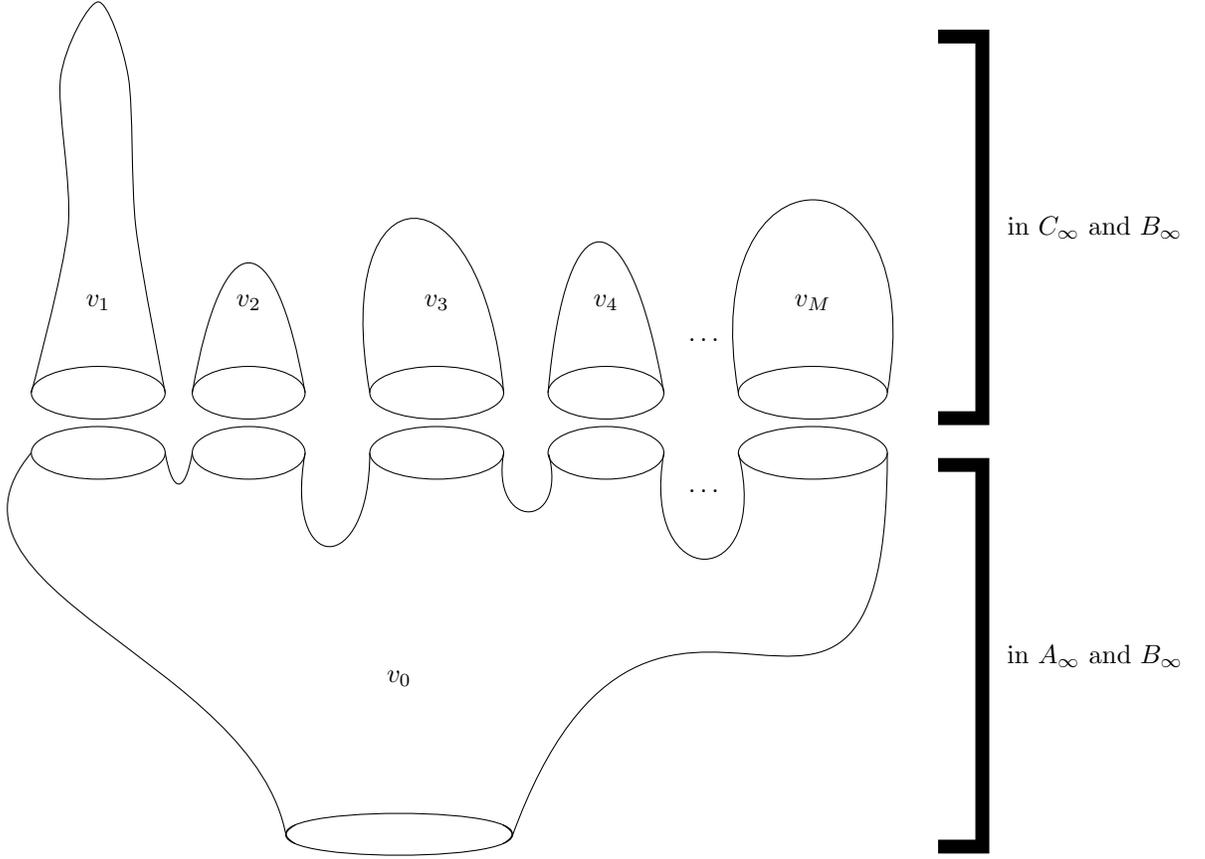

\begin{lemma}\label{emma} The identified curves mentioned in the bulleted list of theorem \ref{happen} must be asymptotic to a Reeb  orbit on $\Sigma$ with $m=0$. 
\end{lemma}
\begin{proof} Because $\Sigma$ in $X$ is of contact type, the symplectic form inherited from $X$ will be exact on a tubular neighborhood of $\partial U$. In fact, let us identify $U$ with a subset of $T^*T^3$, as above, so that the exact form on $\partial U$ has a primitive whose integral over a Reeb orbit of $\Sigma$ in homology class $(k,\ell, m)$ is given by the following formula (in an arbitrarily large range, and up to a small correction due to the smoothing $\Sigma$).
\[ \mathcal{A} = \frac{\epsilon}{2}|k| + \frac{(2d+1)\epsilon}{2}|\ell| + \frac{S}{4}|m|. \]

This action formula will factor into the formulas for the area of curves in $C_\infty$ and $A_\infty$. Because our original holomorphic plane had area at most $da+b << S$, we see that the special curve in ${A_{\infty}}$ cannot have any positive ends asymptotic to Reeb orbits with $m\neq 0$. By the matching conditions, the negative ends of limiting curves in ${C_{\infty}}$ also must be asymptotic to Reeb orbits in class $(k,\ell, 0)$.  
\end{proof}

\section{Analysis of Limiting Curves in ${C_\infty}$}\label{fred}
Now that we understand the procedure of stretching the neck, along with the limiting curves that can result, let us focus on a limiting curve in the part ${C_\infty}$. In this section, let $u$ denote any of the curves $v_1, ... ,v_M$ from the proof of theorem \ref{happen}.  We showed that such a curve $u$ will in fact be a $J$-holomorphic plane asymptotic to a single hyperbolic Reeb orbit on its negative end.  
Using lemma \ref{emma}, the index formula (\ref{generalindex}) becomes 
\begin{equation}\label{virtind} Index(u) =0+ 2c_\tau(u) - CZ(k,\ell) + \frac{1}{2}(2) = 4d_1+4d_2-2(k+\ell), \end{equation}
where $c_\tau$ denotes the relative first Chern class of $u$ and $CZ$ denotes the Conley-Zehnder index of the Reeb orbit on the negative end of $u$. In particular
\begin{equation}\label{virtind2} Index(u) = 4d+4-2k-2\ell \end{equation}
in the case when the curve $u$ has bidegree $(d,1)$.

\begin{lemma}\label{virtual}
All limiting planes $u\in\{v_1,...,v_M\}$  must have virtual index $0$ or $2$.
\end{lemma}
\begin{proof} First, the formula (\ref{virtind}) shows that the virtual index must be an even integer. Second, theorem \ref{happen} gives that the curves $v_0,...,v_M$ must have non-negative index. We noted in the previous section that the Reeb orbits on $\Sigma$ come in $2$-dimensional families, and the limiting curves $v_1,...,v_M$ are asymptotic to these orbits at the negative end. If the virtual index of $u$ were to exceed the dimension of the family of asymptotic orbits, then other curves in the holomorphic building would necessarily have negative index, which is also precluded by theorem \ref{happen}. This leaves only the possibilities of $0$ or $2$ for the virtual index.
\end{proof}

The following two corollaries (\ref{kindex} and \ref{kindex2}) and lemma \ref{elephant} apply to the identified components $v_1,...,v_M$ in $C_\infty$.

\begin{cor}\label{kindex}Any limiting component in homology class $(k,0,0)$ for some negative integer $k$, must actually have $k=-1$ and the virtual index must be $2$ and the bidegree must satisfy $d_1=d_2=0$\end{cor} 
\begin{proof}
Assume $\ell=0$ and $k<0$. Then the index formula (\ref{virtind}) simplifies to
\[ 4(d_1+d_2) + 2|k| = 0 \text{ or } 2,\]
using lemma \ref{virtual}. Now any positive bidegree $(d_1,d_2)\neq (0,0)$ would contribute a multiple of $4$ to the index and contradict the above equation. So we must have $d_1=d_2=0$. Consequently $k$ must be $-1$ and the index must be $2$. 

\end{proof}

\begin{cor}\label{kindex2}Any limiting curve in homology class $(k,0,0)$ for some positive integer $k$ must have nonzero bidegree, meaning $(d_1,d_2)\neq (0,0)$\end{cor} 
\begin{proof} The argument is the same as for corollary \ref{kindex}

\end{proof}

\begin{lemma}\label{elephant} Assume that $x>2$, $n=3$, and $a<2$. If $b<x$, then $\ell \leq 0$ in $\mathbb{Z}$. 
\end{lemma}
\begin{proof}
It will only be necessary to prove this assertion in the cases when the bidegrees of $u$ satisfy $(d_1,d_2) = (0,0)$ or $(d_1,d_2) = (d,1)$ or $(d_1,d_2)=(d,0)$ for some $d>0$ in $\mathbb{Z}$. The following analysis will apply to all cases. 

We note that up to a correction of order $\epsilon$, the area of a curve $u$ of bidegree $(d_1,d_2)$ is given by the formula
\begin{equation}\label{area} Area(u) = d_1a + d_2b - (1k+x\ell) \geq 0, \end{equation}
as all $J$-holomorphic curves have non-negative area. We subtract from this area inequality one-half of the virtual index above to find 
\begin{equation}\label{mainin}
d_1(a-2) + d_2(b-2) - \ell(x-1) \geq -1
\end{equation}
which will be the main inequality for this proof. Here we have used Lemma \ref{virtual} to restrict the virtual index.

\textbf{Case 1:} If $(d_1,d_2) = (0,0)$, then the main inequality reduces to 
\[ -\ell(x-1) \geq -1 \qquad \Longrightarrow \qquad \ell\leq 1/(x-1). \]
Since we have have that $x-1 > 1$, the integrality of $\ell$ forces $\ell \leq 0$. 

\textbf{Case 2:} If $(d_1,d_2) = (d,1)$, then the main inequality (\ref{mainin}) reduces to 
\[ d(a-2) + (b-2) - \ell(x-1) \geq -1 \qquad \Longrightarrow \qquad \underbrace{d(a-2)}_{\text{negative quantity}} + \underbrace{(b-1)}_{\text{small quantity}} \geq \ell\cdot\underbrace{(x-1)}_{\text{positive}} \]
Since $b<x$ we have $\frac{b-1}{x-1}<1$. Hence
\[ \ell \leq \left\lfloor \frac{d(a-2)}{(x-1)} + \frac{(b-1)}{(x-1)} \right\rfloor\leq 0.\]
Repeat this proof, replacing $d_2=1$ with $d_2=0$, to see that the only remaining case is easier than case 2.
\end{proof}

Finally, we exhibit precisely which finite energy planes occur as limits in ${C_\infty}$. \begin{thm}\label{seeinfinity}
Assume $a<2$ and $b<x$. The limiting curves in ${C_\infty}$ consist of 
\begin{itemize}
\item A single plane of bidegree $(d,1)$ which is negatively asymptotic to a Reeb orbit in $\Sigma$ that is in homology class $(2d+1,0,0)$ and which has virtual index $2$;
\item A collection of $(2d+1)$-many finite energy planes of bidegree $(0,0)$, each of which is negatively asymptotic to a Reeb orbit in $\Sigma$ that is in homology class $(-1,0,0)$. Moreover, these planes have virtual index $2$.
\end{itemize}
\end{thm}
\begin{proof}
The proof of this theorem will use the notation of theorem \ref{happen}, and will rely on the lemmas and corollaries above. Consider the $J$-holomorphic planes $v_1,...,v_M$ of theorem \ref{happen}. We label the homology class of the negative end of the plane $v_j$ by the ordered triple $(k_j,\ell_j,m_j)$, for $1\leq j \leq M$. Similarly, we label the bidegree of the plane $v_j$ by the ordered pair $(d_1^j, d_2^j)$, although some of these degrees may be zero. By definition, the bidegree is a bilinear pairing, and we must have
\begin{equation}\label{bidegreedy} \sum_{j=1}^M (d_1^j,d_2^j) = (d,1),\end{equation}
which is the bidegree of the original curve, before stretching the neck.
Since each $d_i^j \geq 0$, we see that the only possibilities for these bidegrees are $(0,0)$ or $(\cdot, 0)$ or $(\ast,1)$ for some non-negative integers $\cdot$ and $\ast$. This shows that we are in the restricted cases that were considered in the proof of lemma \ref{elephant}, and the conclusions of that lemma apply here. Combining this result with the equation (\ref{suml}) implies that $\ell_j = 0$ for $1 \leq j \leq M$. We also know that all $m_j=0$ by lemma \ref{emma}. Hence we may label the homology classes of the negative ends of the planes $\{v_j\}$ by the ordered triple $(k_j,0,0)$, for $1\leq j \leq M$. Notice that, by construction, no $k_j = 0$, because we do not consider the homology class $(0,0,0)$ to describe a valid Reeb orbit. We do have a condition (\ref{sumk}) on the $k_j$, which will allow us now re-index the curves $\{ v_1,...,v_M\}$ so that
\begin{itemize}
\item the subset $\{v_1,...,v_t\}$ consists of curves $v_j$ with $k_j > 0$, and
\item the subset $\{ v_{t+1},...,v_M\}$ consists of curves $v_j$ with $k_j < 0$.
\end{itemize}

The remainder of this proof will show that these two bullet points correspond to the bullet points in the statement of the theorem. We immediately notice that corollary \ref{kindex} applies to the set $\{ v_{t+1},...,v_M\}$, implying that $k_j = -1$ for these curves, and the bidegrees vanish. Since we have organized the curves according to the sign of the integer $k_j$, we may decompose the summation (\ref{sumk}) into a more subtle result. We have
\begin{equation}\label{sumk3} 0 = \sum_{j=1}^t k_j - \sum_{j=t+1}^M |k_j| \Longrightarrow \sum_{j=1}^t k_j =  \sum_{j=t+1}^M |k_j| =  \sum_{j=t+1}^M |-1| = M-t. \end{equation}
This calculation also implies that
\begin{equation}\label{doublekay}
\sum_{j=1}^M |k_j| = 2\sum_{j=1}^t k_j  = 2(M-t).
\end{equation}

By definition, the bidegree is a bilinear pairing. The vanishing of certain bidegrees further implies that 
\begin{equation}\label{bidegreedy2} \sum_{j=1}^t (d_1^j,d_2^j) = (d,1),\end{equation}
which is slightly sharper than (\ref{bidegreedy}). We re-index yet again so that $v_1$ is the curve with $d_2^1 = 1$, and all other $v_j \in \{v_2,...,v_t, v_{t+1},...,v_M\}$ have $d_2^j=0$. Corollary \ref{kindex2} shows that $v_j \in \{v_2,..,v_t\}$ cannot also have $d_1^j = 0$. In other words, all the curves in the set $\{v_1,...,v_t\}$ have some nonzero (strictly positive) component of bidegree.

On the other hand, we have the curve $v_0$ with $M$ positive ends and a single negative end asymptotic to the $(2d+1)$-times cover of the Reeb orbit $\gamma_1$. The area of $v_0$ is given by Stokes' Theorem.
\[ Area(v_0) = \frac{\epsilon}{2}\sum_{j=1}^M |k_j| - (2d+1)(\epsilon - \delta_1). \]
Using formula (\ref{doublekay}) and the positivity of area, we get
\[ M-t = \frac{1}{2} \sum_{j=1}^M |k_j| > (2d+1)\left( \frac{\epsilon - \delta_1}{\epsilon} \right).\]
Since $\delta_1$ may be chosen to be much smaller than $\epsilon$, and since $M-t$ is integral, we find that $M-t \geq 2d+1$. The curves $v_1,...,v_t$ have virtual Fredholm index given by formula (\ref{virtind2}). Suppose, towards a contradiction, that the sum of all the Fredholm indices of these curves were zero. Then, by non-negativity, each curve must have Fredholm index zero. Since curve $v_j$ has bidegree $(d_1^j,d_2^j)$ and has negative end in class $(k_j,0,0)$, we apply the index formula (\ref{virtind}) to find
\[ 0 = \frac{1}{2} Index(v_j) = 2d_1^j + 2d_2^j - k_j \Longrightarrow k_j = 2d_1^j + 2d_2^j.\]
In a similar way, we can sum the areas of these curves to find
\[ \sum_{j=1}^t Area(v_j) = da+b - \sum_{j=1}^t k_j = da+b - 2d - 2 = d(a-2) + (b-2).\]
The assumption $a<2$ makes the term $d(a-2)$ negative, and we are free to choose $d$ large enough to ensure that $d(a-2) + (b-2)<0$, contradicting positivity of area. We must conclude that not all the curves $v_1,...,v_t$ have Fredholm index zero. In other words, 
\[ 0 < \sum_{j=1}^t Index(v_j) = 4d+4 - 2\sum_{j=1}^t k_j = 4d+4-2(M-t).\]
Consequently, $M-t < 2d+2$. We have now shown
\[ 2d+1 \leq M-t < 2d+2.\]
By integrality, we must have $M-t = 2d+1$. 
Substituting this result into the Index computation above gives
\begin{equation}\label{sumkay}
\sum_{j=1}^t Index(v_j) = 4d+4-2(2d+1) = 2.
\end{equation}
By lemma \ref{virtual}, the indices of $v_1,...,v_t$ can be either $0$ or $2$. This leaves only one possibility: a single curve among the collection has index $2$, and the rest have index zero. 

We claim that $v_1$ is the curve with virtual Fredholm index $2$. The proof will be very similar to the argument in the preceding paragraph. Assume, towards a contradiction, that $Index(v_1)=0$. Then
\[ 0 = \frac{1}{2}Index(v_1) = 2d_1^1 + 2(1) - k_1 \Longrightarrow k_1 = 2d_1^1 + 2.\]
Substituting this into the area equation of this curve gives
\begin{equation}\label{area1}
Area(v_1) = d_1^1a+ 1b - k_1 = d_1^1(a-2) + (b-2)
\end{equation}
The area of the remaining curves must therefore be
\begin{equation}\label{arearest}
\sum_{j=2}^t Area(v_j) = (d-d_1^1)a - \sum_{j=2}^t k_j = (d-d_1^1)a - 2d - 1 + 2d_1^1 + 2 = (d-d_1^1)(a-2) + 1
\end{equation}
Now, we compare equations (\ref{area1}) and (\ref{arearest}). We showed previously that we may choose $d$ large enough so that $d(a-2)+(b-2) < 0$. But since $d_1^1$ is a summand of $d$, increasing $d$ might increase $d_1^1$, which would make equation (\ref{area1}) negative; or increasing $d$ might not increase $d_1^1$ in which case equation (\ref{arearest}) becomes negative. Either case will contradict the positivity of area. We must conclude that $Index(v_1)=2$.

Next we show that, in fact, $t=1$. Assume, towards a contradiction, that $t \geq 2$. Then all of the curves $\{ v_2,...,v_t\}$ must have index zero, because of equation (\ref{sumkay}). Consequently
\[  0 = \frac{1}{2}Index(v_j) = 2d_1^j - k_j \Longrightarrow k_j = 2d_1^j \qquad \text{ for } 2 \leq j \leq t. \]
But any such curve $v_j$, $2\leq j \leq t$ will have
\[ Area(v_j) = d_1^ja-k_j < 2d_1^j - 2d_1^j = 0, \]
since $a<2$. Again, this contradicts the positivity of area, and we must conclude that $t=1$. 

This plane $v_1$ is the only $J$-curve to touch the divisor $L_\infty$, and it must have bidegree $(d,1)$, by construction. It must also have $k_1 = M-t = 2d+1$, by above work. This proves the first bullet point of the theorem statement. Since there are a total of $M = 2d+1 +t = 2d+2$ planes in $C_\infty$, we know that the remaining $2d+1$ curves $\{ v_2, ..., v_M\}$ do not meet the divisor at infinity, and we already showed that these planes are in homology class $(-1,0,0)$. Equation (\ref{virtind})  shows that each of $\{ v_2, ..., v_M\}$ has index $2$. This proves the second bullet point of the theorem.

\end{proof}

\section{Analysis of the Special Curve}\label{special}

Looking back at Theorem \ref{happen}, we see that the negative ends of all curves in ${C_\infty}$ must match with the positive ends of the so-called special curve, $v_0$. These matching conditions allow us to use Theorem \ref{seeinfinity} to completely characterize the Reeb orbit asymptotics of the special curve. In this section, we will complete the proof of the main theorem.

\begin{cor}\label{ainfinity}
The special curve $v_0$ must be positively asymptotic to a Reeb orbit in $\Sigma$ in homology class $(2d+1,0,0)$ and positively asymptotic to an additional $(2d+1)$-many Reeb orbits, counting multiplicity. The special curve must be negatively asymptotic to a $(2d+1)$-times cover of the short orbit $\gamma_1$ of $E$. \\
The area of this special curve is
\[ \frac{\epsilon}{2}\left(2d+1 \right) - (2d+1)(\epsilon - \delta_1), \]
Finally, the virtual Fredholm index is zero.
\end{cor}
\begin{proof}
The claim about Reeb orbits is just a restatement of theorem \ref{seeinfinity} and the construction in theorem \ref{bigcurve}. The area calculation is Stokes' theorem, and was explained in the proof of theorem \ref{seeinfinity}. 

The curve that was constructed in theorem \ref{bigcurve} has virtual Fredholm index zero. This curve was stretched to give curves $v_0, ..., v_{2d+2}$. Our special curve is $v_0$, and it can be matched along its positive end to an additional $2d+2$ curves of index $2$ whose ends are negatively asymptotic to Reeb orbits in $2$-dimensional families. This matching is explained in theorem \ref{happen}. The matching procedure imposes constraints on the Fredholm indices. The sum of the indices of all the constituent curves $v_0,...,v_{2d+2}$ must equal the sum of the parent curve plus the sum of the dimensions of orbits at the sites of matching. This implies
\[ \underbrace{2 \cdot (2d+2)}_{\text{dim. of matching}} + \underbrace{0}_{Index(\text{parent})} = Index(v_0) + \underbrace{2\cdot(2d+2)}_{Index(v_1)+...+Index(v_M)}.\]
Consequently $Index(v_0) = 0$. 
\end{proof}

Finally, we are in a position to prove the main theorem of this paper.
\begin{thm}(Main theorem, remaining case) 
Suppose  that $x\geq 2$, $n\geq 3$, and $a<2$. If there exists a symplectic embedding \[P(1,x)\times\R^{2n-4} \hookrightarrow P(a,b) \times \R^{2n-4}\] then $b \geq x$.
\end{thm}

\begin{proof}
If such an embedding exists, we may assume, without a loss of generality, that $x>2$. Assume, towards a contradiction, that $b<x$. Then the results of  lemma \ref{elephant} and theorem \ref{seeinfinity} and corollaries \ref{kindex} and \ref{ainfinity} all hold.


We showed in equation (\ref{sumk3}) that
\begin{equation}\label{sumk4}  \sum_{k_i < 0}|k_i| = 2d+1.\end{equation}
As a result of matching imposed by theorem \ref{seeinfinity}, our limiting building contains $2d+1$ planes asymptotic to Reeb orbits on $\Sigma$ in homology class $(-1,0,0)$ and having area $1+O(\epsilon)$. The total area of this limit is at most the area of the plane before stretching, which was computed in section \ref{neck}. This, together with (\ref{sumk4}) gives
\[ da+b - (2d+1)(\epsilon - \delta_1) \geq (2d+1)(1+O(\epsilon)).\]
Dividing by $2d+1$ and using $a<2$ gives
\[ \frac{b}{2d+1} + \frac{2d}{2d+1} - (\epsilon - \delta_1) >  \frac{da}{2d+1} + \frac{b}{2d+1} - (\epsilon - \delta_1) \geq 1+O(\epsilon). \] 
This will give a contradiction when $\epsilon$ is sufficiently small, because we initially assumed $b << 2d+1$.
\end{proof}


\section{The Proof of Theorem \ref{bigcurve}}\label{curveproof}

Now that we have the machinery of neck-stretching, we can prove theorem \ref{bigcurve}. This proof will not rely on any of the proofs of prior sections, just the definitions.  We restate the theorem here with slightly updated terminology.

\begin{definition} Let us say that $2d+1$ distinct points of a symplectic manifold are \emph{generic relative to $J$} if there's no closed $J$-holomorphic curve of index less than $4k+2$ passing through $2k+1$ of the points, for all $0\leq k\leq d$.  Here the notion of curve includes finite energy curves in the completion of the ellipsoid, described below.
\end{definition}

\begin{definition} Fix a set of points $p_1,...,p_{2d+1}$ which are generic with respect to the almost complex structure $J_0$.  A deformation $t\mapsto J_t$, $t\in[0,1]$, of the almost-complex structure is a \emph{generic deformation} if for all $t \in [0,1]$ a curve of index $2k$ passes through $k+1$ of the constraint points, $0\leq k \leq d$. If a deformation of $J_0$ is generic, then we call any $J_t$ in this deformation a generic almost-complex structure.
\end{definition}

\begin{theorem}
In a completion of $X \setminus E$ there exist $J$-holomoprhic planes of bidegree $(d,1)$, of virtual Fredholm index zero, and which are asymptotic to $\gamma^{(2d+1)}$. Such curves persist under scaling of the ellipsoid and under generic  deformations of the almost-complex structure.
\end{theorem}

We begin with the completion of $X\setminus E$, and we observe that setting $z_3=0$ gives an almost-complex submanifold \[\overline{Y} : = \CP{1}(a)\times \CP{1}(b) \setminus E(\epsilon - \delta_1, (2d+1)(\epsilon-\delta_2)).\] 
Actually, we will assume that the almost-complex structure is chosen so that $\overline{Y}$ is a complex submanifold of $X\setminus E$.
We abbreviate $\tilde{E}:=E(\epsilon - \delta_1, (2d+1)(\epsilon-\delta_2)) \subset \CP{1}(a) \times \CP{1}(b)$.
Here is a brief outline of the argument that follows.
\begin{itemize}
\item We produce a curve $C$ in $\CP{1}(a)\times\CP{1}(b)$ using classical techniques.
\item We stretch the neck along the boundary of $\tilde{E}$ to analyze the asymptotics of the curve $C$.
\item We argue that the curve $C$ persists under inclusion into the higher-dimensional space $X\setminus E$, and under scaling of the ellipsoid. 
\end{itemize}

Because of the persistence argument, it suffices to first produce curves in a completion of $\overline{Y}$ (which will have a cylindrical end). This will be the focus of the fourth proposition below. 
Recall the definition of bidegree in section \ref{polyd}.  Notice that the ratio of capacities in $\tilde{E}$, namely
\[\theta: = \frac{\epsilon - \delta_1}{(2d+1)(\epsilon - \delta_2)} \]
is irrational, because of the assumptions made in (\ref{oneone}). The irrationality of $\theta$ ensures that the Reeb orbits of $\partial\tilde{E}$ are isolated, and there will be a short orbit. This number $\theta$ will also be the \textit{monodromy angle} of the short orbit $\gamma_1$ on $\partial\tilde{E}$, as defined in section \ref{ell} (see also Lemma \ref{czgam}). Further, note that the Reeb orbits on $\partial \tilde{E}$ are elliptic and they exist in $0$-dimensional families. 
Theorem \ref{bigcurve} will be proved in the following five propositions.

\bigskip

\begin{prop}\label{prop1} In $\CP{1}(a)\times\CP{1}(b)$ there exists a closed curve of bidegree $(d,1)$ passing through $2d+1$ generic points. 
\end{prop}
\begin{proof}
We first produce a closed embedded curve for the standard $J$, which is denoted $J_0$. Such a $J_0$-curve can be viewed as a graph of a meromorphic function $\CP{1}(b)\to \CP{1}(a)$ which projects to a curve of degree $d$ in $\CP{1}(a)$ and which is biholomorphic to the sphere when projected onto $\CP{1}(b)$.  Let us say that the meromorphic function being graphed is of the form $f/g$, where both $f$ and $g$ are polynomials of degree $d$. Such polynomials are each specified by $d+1$ parameters, and we must factor out a common scale when dividing $f$ by $g$. Hence the quotient $f/g$ has $(d+1)+(d+1)-1 = 2d+1$ parameters. This shows that we can produce a $J_0$-curve through $2d+1$ given points of $\CP{1}(a)\times \CP{1}(b)$. This curve is non-singular, because it is of degree $1$ in the second factor. 

We present an alternative point of view, by showing that curves in $\CP{1}(a)\times\CP{1}(b)$ passing through $2d+1$ constraint points are in bijective correspondence with curves in the $(2d+1)$-fold blow-up of  $\CP{1}(a)\times\CP{1}(b)$. This bijection will be useful for future computations. Fix a generic, ordered set of points $\{p_1,...,p_{2d+1}\} \subset \CP{1}(a)\times\CP{1}(b)$, and let $\hat{X}$ denote the $(2d+1)$-fold blow-up of $\CP{1}(a)\times \CP{1}(b)$ at those points. Let $\pi \colon \hat{X} \to \CP{1}(a)\times\CP{1}(b)$ denote the projection map. Denote the exceptional divisors by $E_1,...,E_{2d+1}$. Let $A$ denote a divisor in $\CP{1}(a)\times \CP{1}(b)$ that passes through all $2d+1$ constraint points. The proper transform gives a unique lift to $\hat{A}$ in $\hat{X}$, and $\hat{A}$ determines an element of $|\pi^{-1}(A) - E_1 - ... - E_{2d+1}|$, which is the projectivization of $\mathscr{O}_{\hat{X}}(\pi^{-1}(A)-E_1-...-E_{2d+1})$. Conversely, given a section on $\hat{X}$ of $\mathscr{O}_{\hat{X}}(\pi^{-1}(A)-E_1-...-E_{2d+1})$, push it down via $\pi$ to obtain a section of the complex line bundle $L(A)$ that is defined away from the $2d+1$ constraint points. We know this complex line bundle exists everywhere by assumption. Use the removable singularity theorem to define this section across the $2d+1$ constraint points. 
This bijection shows that $H^0(\hat{X},\mathscr{O}(\pi^*(L(A))))\cong H^0(X,\mathscr{O}(L(A))) = H^0(\mathscr{O}(d_1)) \otimes H^0(\mathscr{O}(d_2))$, for curves that have bidegree $(d_1,d_2)$. The latter space has dimension $(d_1+1)(d_2+1)$ by the K\"unneth formula. Based on the algebro-geometric definition of genericity of points, we expect a collection of points to be generic if no curve of bidegree $(0,1)$ or $(1,0)$ passes through more than one of the points, along with higher-degree restrictions.  The more general requirement is that no curve of bidegree $(d_1,d_2)$ may pass through more than $(d_1+1)(d_2+1) - 1$ of the points. When we consider the special case of bidegree $(d_1,d_2) = (k,1)$, we have that $(k+1)(1+1)-1 = 2k+1$. Furthermore, the index formula (\ref{closedindex}) implies that a closed genus zero curve of bidegree $(k,1)$ will have virtual Fredholm index $-2+4(k+1) = 4k+2$. This shows the parallel between the above definition of \textit{generic relative to $J$} and the usual definition of genericity of points. The constraints in the definition are the best possible, because one would expect curves passing through exactly $2k+1$ points to generate a complex of nullity zero. The bijection described in this paragraph gives a formula of $dim[H^0(X, \mathscr{O}(\mathscr{L}_A)] = (d+1)(1+1)=2d+2$ for curves of bidegree $(d,1)$. Given $2d+1$ generic points, there is a unique section, up to scale, that vanishes at the generic points, because $(2d+2)-(2d+1) = 1$. This shows that there is a unique curve, up to scale, passing through the $2d+1$ given points. This curve will be generically embedded if it avoids a proper discriminant locus, but we will prove it's embedded using the adjunction formula.

Now we view this newly-constructed curve for the standard $J_0$ as a map
\[ u \colon S^2 \to \CP{1}(a)\times \CP{1}(b) \]
(not as a graph), and we apply the four-dimensional adjunction inequality (which is formula 2.7 in  \cite{ECHnotes}). Let $[C]$ denote the class of the image of $u$. Then
\[ \chi(\text{domain of } u) + [C]\cdot [C] -\langle c_1(T(\CP{1}(a)\times \CP{1}(b))),[C] \rangle  \geq 0,\]
with equality if and only if the curve $u$ is embedded. In our case, this reduces to
\begin{align*} 2 + \langle(d,1),(d,1) \rangle - \langle (2,2),(d,1)\rangle &= \\
2+ [d^2(f_1^2) + 2d (f_1\cdot f_2) + 1(f_2^2)] - [2d(f_1^2)+(2d+2)(f_1\cdot f_2) + 2(f_2^2)] &= 0,
\end{align*}
where in this calculation $f_i^2$ denotes the square of the fiber in the $i^{th}$ variable. Being an embedded sphere, such a curve is automatically regular (using Lemma 3.3.3 of \cite{MSal}).
\end{proof}

The last formula in the above proof is  purely topological, and we shall explain in Proposition \ref{prop3} how the result will persist under generic deformations of the almost-complex structure. This will show that curves for any $J$ are embedded.

Specifically, for any $J$, let us define a moduli space
\[ \mathcal{M}(J,p_1,...,p_{2d+1}) := \left\{ (f,(y_1,...,y_{2d+1})) \left| \,\overline{\partial}_Jf=0, f(y_i)=p_i, f\text{ has bidegree } (d,1) \right. \right\}/\sim, \]
of closed, constrained $J$-holomorphic curves identified under $\sim$ by reparameterization of the domain. In the standard case, Proposition \ref{prop1} shows that $\mathcal{M}(J_0,p_1,...,p_{2d+1})$ is nonempty, and representative curves are embedded and automatically regular. To extend this claim to non-standard $J$, we need to argue that the adjunction formula applies.

We vary the almost complex structure in time, to get a one parameter family $J_t$, $t\in[0,1]$, with $J_0$ denoting the standard complex structure. We need to show that the set of points which were chosen to be generic for $J_0$ remain generic for $J_t$, $t\in[0,1]$. Since the constraint points remain fixed for all time, we may and shall choose these $2d+1$ points to be in the interior of $\tilde{E}$. 

\begin{prop}\label{prop2} There is a second category set of maps $t\mapsto J_t$ such that the points $\{ p_1,...,p_{2d+1}\}$ are generic relative to $J_t$ for all $t$. \end{prop}
For a proof of this proposition, see lemmas 2.4 through 2.6 and corollary 2.7 of \cite{newobs}, replacing $\CP{2}$ with $\CP{1}\times \CP{1}$.

In other words, it is possible to choose a deformation $J_t$ of the standard structure $J_0$ that is \textit{generic}, as defined above. We assume such a deformation is given for the remainder of this section. Now we may update the definition of $\mathcal{J}^\star$. In addition to the list of properties at the end of section \ref{polyd}, we need $J_t \in\mathcal{J}^\star$ to be such that $\overline{Y}$ is a complex submanifold of $X\setminus E$, and compatible with the stretching along $\partial \tilde{E}$, and generic. 

Recall that a holomorphic curve which is not multiply-covered is called \textit{simple.}

\begin{prop}\label{prop3} For any $t\in[0,1]$, the moduli space 
\[ \mathcal{M}_t: = \mathcal{M}(J_t,p_1,...,p_{2d+1}) =\left\{ (f,(y_1,...,y_{2d+1})) \left| \,\overline{\partial}_{J_t}f=0, f(y_i)=p_i, f\text{ has bidegree } (d,1) \right. \right\}/\sim,  \]
of closed $J_t$-holomorphic curves passing through the generic points $\{ p_1,...,p_{2d+1} \}$ is compact. Curves in this moduli space are identified under reparameterizations of the domain. $\mathcal{M}_t$ has virtual dimension zero. Furthermore, for any $t\in[0,1]$, a representative of $\mathcal{M}_t$ is not nodal, is embedded, is simple, and has ECH index zero. \end{prop}

\begin{proof} The proof of the first claim in this proposition is a standard Gromov compactness argument. Here we are deforming a closed, embedded curve, and a limit of such curves will either be embedded or will bubble. Any bubbling would be of codimension $2$, which is precluded for generic deformations of the almost-complex structure, because some component of the bubble tree would have index $-2$.

This compactness implies that we can apply the adjunction inequality (above) to $J_t$-holomorphic curves. The conclusion is that for any $t$, the $J_t$-curve that is constrained to pass though a set of $(2d+1)$ $J_t$-generic points will be embedded, and not nodal. 

We prove the claim of simple-ness by contradiction. Suppose that a curve $C\in \mathcal{M}(J_t, p_1,...,p_{2d+1})$ of bidegree $(d,1)$ were multiply-covered. Then $C$ must cover some underlying curve $K$ which has bidegree $(i,j)$ for $i<d$ and $j<1$, since multiple-coverings will multiply the intersections with $L_\infty$. This underlying curve will have Fredholm index $2c_1-2 = 4(i+j) - 2$. This Fredholm index is strictly less than $2(2d+1)$. Hence, for generic $J$ the curve $K$ cannot pass through the $2d+1$ constraint points (with each intersection having codimension $2$). This contradicts our construction of $C$ and the genericity of points. We must conclude that $C$ is simple. (Moreover, it is a general fact that a limiting curve is either multiply-covered or is embedded with possibly finitely many, denoted $\delta$-many, point singularities. We shall make use of this fact later.)

We showed in proposition \ref{prop2} that the constraint points $p_1,...,p_{2d+1}$ are generic relative to $J_t$. The virtual dimension of the moduli space $\mathcal{M}_t$ is computed using the virtual Fredholm index of a representative curve. We are claiming that both the virtual Fredholm index and the ECH index of a representative curve equal zero. To compute the virtual index, we use the formula (\ref{closedindex}) for a closed curve of genus zero. Without any constraint points, a curve $C$ is closed, embedded, and has virtual Fredholm index 
\[ Index(C) = (2-3)(2-0) + 2(2d+2) = 4d+2.\] 
The moduli space of curves corresponding to this Fredholm index can only be counted if $2d+1$ generic constraint points are specified. As noted above, each constraint point has codimension $2$. 
Hence the curve $C$, when constrained to pass through a generic set of  $2d+1$ points, will have virtual Fredholm index zero. Finally, we note that a closed, embedded curve of virtual Fredholm index zero must also have ECH index zero. This fact follows from equation (2.10) of \cite{ECHnotes}.
\end{proof}

The reason for this computation is as follows. We are now going to stretch the neck  along $\partial E$, which will create, as an intermediate step, almost-complex structures $J^N$ on the manifolds stretched to length $N$ (see theorem \ref{buildingE} below). Initially (in Proposition \ref{prop1}) we produced curves for the standard, integrable, almost-complex structure $J_0$. The point of Proposition \ref{prop3} is to show that, after deforming to $J_t := J^N$, the curves persist, and similarly when we take a limit of $J_t$. For this neck stretching we can re-scale the compact interval as $-\infty \cup (-\infty, 0] \cong [0,1]$. Recall that we assumed that the constraint points are in the interior of $\tilde{E}$.
Now, we can stretch the neck of $\CP{1}(a) \times \CP{1}(b)$ along $\partial \tilde{E}$, and use the following SFT theorem to analyze the limiting building.

\begin{thm}\label{buildingE} (See 10.6 of \cite{sft}) Fix $2d+1$ constraint points in $\tilde{E}$ and fix a $J$ in $\mathcal{J}^*$. For each $N \in \mathbb{N}$, let $u_N$ be a $J^N$-holomorphic curve passing through the $2d+1$ constraint points. Fix a representative $f_N$ for each $u_N$. Then there exists a subsequence of the $f_N$ which converges to a holomorphic building $\textbf{F}$. The domain of $\textbf{F}$ is a nodal Riemann sphere $(S,j)$ with punctures, and the building can be described as a collection of finite energy holomorphic maps from the collection of punctured spheres $S \setminus \{\text{nodes}\}$ into one of the following three completions:
\[ A_\infty^E : = \tilde{E} \cup_{\partial\tilde{E}} (\partial{\tilde{E}} \times [0,\infty)) \qquad \text{ with form } d(e^t\lambda),\, t\in [0,\infty),\]
\[ B_\infty^E : = \partial\tilde{E} \times \R,\qquad  \text{ with form } d(e^t\lambda), \, t\in \R ,\]
or
\[ C_\infty^E : = \overline{Y} \cup_{\partial\tilde{E}} (\partial\tilde{E} \times (-\infty,0]) \qquad \text{ with form } d(e^t\lambda),\, t\in (-\infty,0]. \]
\end{thm}

It will be convenient to view the completion of $\overline{Y}$ from theorem \ref{buildingE} as a symplectic null-cobordism. For notational brevity, we set
\[ Y:= C_\infty^E =\left(\, (-\infty,0]\times \partial \tilde{E} \,\right)\cup_{\partial \tilde{E}} \overline{Y}.\]
Then $Y$ is a completion of $\overline{Y}$ with $\partial Y = \partial\tilde{E} \sqcup \emptyset$.  In addition, this notation matches the notation of section 4.1 of \cite{partitions}.

As in the discussion following Theorem \ref{building}, curves in $Y$ can have only negative asymptotic ends. For this stretching, we will not focus much on curves in the ellipsoid or in the symplectization, $B_\infty^E$, other than to note that we have some limiting curves inside and outside the neck. All curves, inside or outside, must have matching asymptotic limits along $\partial\tilde{E}$. A limit of embedded curves must either consist of curves that are embedded or curves that are multiply covered. The problem here is to show that stretching the parent curve $C$ of bidegree $(d,1)$ does \textit{not} create limiting planes of strictly smaller bidegree. Let $\tilde{C}$ denote any connected component in $Y$. We now argue the following.

\begin{prop}\label{prop4} The holomorphic building that results from stretching the neck along $\partial \tilde{E}$ consists of simple, embedded components with ECH index and virtual Fredholm index zero. Each component in the highest level, $Y$, of the building is a holomorphic plane, $\tilde{C}$, that must be negatively asymptotic to the short Reeb orbit on $\tilde{E}$, not the long Reeb orbit.
\end{prop}

\begin{proof} We showed in Proposition \ref{prop3} that any limiting curves must be simple and not nodal.  As mentioned above, a limiting curve that is not multiply covered must be embedded except at finitely many (here $\delta=0$ many) nodes. We conclude that such a simple curve is embedded.

Next we compute the ECH index and the virtual Fredholm index of $\tilde{C}$. In Proposition \ref{prop3}, we showed that the parent curve, before stretching, had ECH index zero. This implies that the sum of the ECH indices of the limiting curves, after stretching, must also be zero. Since the ECH index is a non-negative quantity, the ECH index of each limiting curve must be zero. Hence we can write $I_{ECH}(\tilde{C})=0$. Now we apply the ECH index inequality from \cite{partitions}. In the case when $\delta = 0$, this inequality says
\begin{equation}\label{ECHineq} Index(\tilde{C}) \leq I_{ECH}(\tilde{C}).\end{equation}
Both of these indices are non-negative, because $\tilde{C}$ is simple. Since we have proved that the ECH index is zero, we can conclude that the virtual Fredholm index is zero. Furthermore, this ECH index inequality is actually an equality, which will be helpful for the last part of this proof.

Next, we prove that $\tilde{C}$ must have bidegree $(d,1)$ and must be negatively asymptotic to the short Reeb orbit on $\tilde{E}$, not the long Reeb orbit. Because of the summation over negative ends in the virtual index formula (\ref{generalindex}), it suffices to consider the case where $\tilde{C}$ has two negative ends. The first is asymptotic to the short Reeb orbit $\gamma_1$ and winds $r_1$ times around this orbit. The second is  asymptotic to the long Reeb orbit, which we shall call $\gamma_2$, and it winds $r_2$ times around this long orbit. The Conley-Zehnder indices of these two (multiply-covered) Reeb orbits are equal to
\[ CZ(\gamma_1^{(r_1)}) = 2r_1 + 2\lfloor r_1\cdot  \theta \rfloor + 1 \approx 2r_1 +0+ 1,   \]
and 
\[ CZ(\gamma_2^{(r_2)}) = 2r_2 + 2\left\lfloor \frac{r_2(2d+1)(\epsilon - \delta_2)}{\epsilon - \delta_1}\right\rfloor + 1 \approx 2r_2 + (4d+2)r_2 + 1,\]
as $d$ becomes large compared to $\epsilon$.
 The index formula (\ref{generalindex}) in dimension $n=4$ with $s^-=2$ negative ends and with bidegree $(d_1,d_2)$ simplifies to 
\begin{align*} Index(\tilde{C}) &= (-1)\cdot(0) +4d_1+4d_2  - CZ(\gamma_1^{(r_1)}) - CZ(\gamma_2^{(r_2)}) + \frac{1}{2}(0+0)\\
& = 4d_1+4d_2 -2 - 2r_1 - 2r_2 - 4dr_2-2r_2\\
&= 4(d_1+d_2) - 2 -2(r_1+r_2) - 2(2d+1)r_2\\
&= 4(d_1+d_2) - 2 - 2(2d+1)(r_2+1)
\end{align*}
where in the penultimate step we used the fact that $r_1+r_2 = 2d+1$, by construction. Since $r_1, r_2 \geq 0$, in order to have $Index(\tilde{C})= 0$ we must have $d_1+d_2 = d+1$ and $r_2=0$. The maximum of each of these degree terms is given by the bidegree of the parent curve, $(d,1)$. Hence for $\tilde{C}$ we must have $d_1=d$ and $d_2=1$ and its negative end(s) must wind along only the short orbit $\gamma_1$.


Finally, it remains to show that $\tilde{C}$ is a holomorphic plane (topologically a disc). We show this by proving that $\tilde{C}$ has only a single negative end asymptotic to a  $(2d+1)$-cover of $\gamma_1$. It turns out that this is determined indirectly in the machinery of ECH, using ``partition conditions,'' which we shall briefly motivate here. Let $\tilde{C}$ denote any component inside of ${Y}$. Again, ${Y}$ is a symplectic null-cobordism ($\partial{Y} = \partial\tilde{E} \sqcup \emptyset$), which implies that $\tilde{C}$ has no positive end, and we assume \textit{a priori} that $\tilde{C}$ has a number of negative ends which are all asymptotic to $\gamma_1$ on $\partial\tilde{E}$ with \textit{total} multiplicity $2d+1$. It is conceivable for $\tilde{C}$ to be topologically a $U$-tube with one end wrapping once around $\gamma_1$ and the other end wrapping $2d$ times around $\gamma_1$. A simple calculation shows that this this configuration is impossible; $\tilde{C}$ must be a $2$-handle. See figure \ref{fig:fig2} for an illustration of this non-example. 
\begin{figure}
  \includegraphics[scale=1.0]{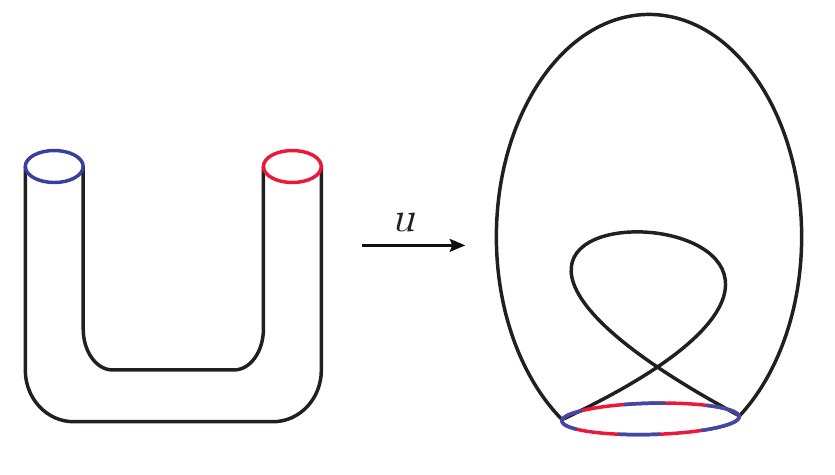}
\caption{Such a $J$-holomorphic curve $u$ is prohibited by the partition conditions.}
\label{fig:fig2}
\end{figure}

The point of the ECH partition conditions is that for an embedded curve in a symplectic cobordism, the multiplicity of its negative end(s) can be computed as a partition $p_j^-(\tilde{C}) = p_{\gamma_j}^-(m_j)$, where the subscript $j$ denotes the $j^{th}$ negative end of $\tilde{C}$, which is asymptotic to a cover of the Reeb orbit $\gamma_j$ with total covering multiplicity $m_j$. The multiplicities of these covers will give a partition of the positive integer $m_j$, hence the name. There is a similar partition condition for the positive end, but in our setup $\tilde{C}$ has no positive end. So we focus on the first negative end of $\tilde{C}$, setting $j=1$. The total covering multiplicity of $\gamma_1$ is $2d+1$ by construction. So we need to compute the partition
\[ p_1^-(\tilde{C}) = p_{\gamma_1}^-(2d+1). \]
The full definition of these ``incoming partition condtions'' for a symplectic cobordism is given in the paper \cite{partitions} with the slightly older notation $p^{in} = p^-$. Here we use the updated notation from \cite{ECHnotes}. In our case, we need only the fact that $p_{\gamma_1}^-(2d+1)$ is entirely defined by the monodromy angle $\theta$. Since $\lfloor (2d+1)\theta \rfloor = 0$, we have that $p_{\gamma_1}^-(2d+1) = (2d+1)$ is the trivial partition.
The relevant theorem from \cite{partitions} says that when the ECH index inequality (\ref{ECHineq}) is an equality, the partition $p_{\gamma_1}^-$ exactly determines the multiplicity of covering(s) of the negative end(s) of the $J$-curve. This equality was proved at the beginning of this proof. The computation that $p^-_{\gamma_1}$ is the trivial partition shows that the curve $\tilde{C}$ can have only a single negative end asymptotic to a $(2d+1)$-cover of $\gamma_1$. Hence, in particular, $\tilde{C}$ is a holomorphic plane, and there is only a single limiting curve in the top levels of the holomorphic building that results from stretching the neck. Since we started with a single curve $C$ (before stretching) and we ended with a single curve $\tilde{C}$ (after stretching), we may as well identify $C = \tilde{C}$ within $Y$, to simplify the notation.
\end{proof}

\bigskip

This completes the first two bullet points in the proof outline above. We now need to explain why such a curve $C$ in ${Y}$ can be included into a completion of $X\setminus E$. First, for comparison purposes, we recall the definition of regularity. In dimension four, an almost-complex structure is regular if the linearized normal Cauchy-Riemann deformation operator is surjective at all somewhere injective curves. In particular, there are no somewhere injective curves of negative Fredholm index, and in dimension four spheres of non-negative Fredholm index automatically have a surjective deformation operator. Above we used the fact that there is a second category generic subset of regular almost complex structures. This discussion leads us to, once again, update the conditions on $\mathcal{J}^\star$. Here we need almost-complex structures on the six-dimensional manifold $X\setminus E$ such that the inclusion $\overline{Y} \to X\setminus E$ (equivalently, their respective completions) is a holomorphic embedding with respect to a compatible almost-complex structure on $Y$ for which a curve exists, satisfying the above propositions. Let $\mathcal{J}^\star$ denote the set of almost-complex structures on $X\setminus E$ with this property, in addition to the properties on $Y$ that follow Proposition \ref{prop2}. We shall have more to say about genericity of the almost-complex structures in a moment, but for now  we emphasize that (after the neck stretching along $\partial \tilde{E}$) a generic almost complex structure $J \in \mathcal{J}^*$ admits no $J$-holomorphic planes of index $-2$ in all possible bidegrees and with all possible windings along the negative end.

\begin{prop}\label{prop5} If such $J$-planes exist in a completion $Y$ of $\overline{Y}$, then they persist under inclusion in the higher-dimensional space $X\setminus E$ (suitably completed). Such a stabilized curve will persist under deformations of the complex structure and scaling of the ellipsoid $E$.\end{prop}
\begin{proof} 
First, consider a completion of $X\setminus E$ using the inclusion embedding of the ellipsoid. We have already constructed in Proposition \ref{prop4} a genus zero $J$-holomorphic curve $C$ in a completion of $\overline{Y}$, for some $J$. We need to argue that the stabilized curve $C \times \{pt\}$ includes into (the completion of) $X\setminus E$ without changing the Fredholm index. We refer back to the general index formula for a genus zero curve, found on page \pageref{bougie}. The following is a standard argument. Stabilizing the curve $C$ obviously does not change the genus, but it does increase $n$ by $1$ and increase each $CZ$ term by $1$ for each factor of $\R^2$ in the stabilization. A useful observation of \cite{newobs} is that we can take genus zero curves having $s^+ = 0$ positive ends, so that the index of the stabilized curve reduces to
\[ Index(C\times\{pt\}) = 2 - 2(s^-) + Index(C).\]
In particular, the index is unchanged under stabilizations when $s^-=1$, which is the case of interest here.  This implies that the curve from Proposition \ref{prop4}, before and after stabilization, will have virtual Fredholm index zero. We temporarily call such a stabilized curve a $J_0$-holomorphic curve, where $J_0\in \mathcal{J}^*$ is a generic almost-complex structure that results from the neck stretching in proposition \ref{prop4}. We now have a moduli space
\[ \mathcal{M}_0 = \{ J_0\text{-holomorphic curves in } X \setminus E \text{ having one negative end asymptotic to } \gamma_1^{2d+1}\}/\sim \]
which has virtual dimension zero, and is nonempty. In this moduli space, curves are identified up to reparameterization of the domain.

After varying $J$ in time, we now have a moduli space
\[ \mathcal{M}_t = \{ J_t\text{-holomorphic curves in } X \setminus E \text{ having one negative end asymptotic to } \gamma_1^{2d+1}\}/\sim. \]
The curves in this moduli space are identified up to reparameterization of the domain, and when $t=0$, we recover the above moduli space $\mathcal{M}_0$. For the remainder of this proof, a generic deformation of the almost-complex structure will be a deformation $t\mapsto J_t \in \mathcal{J}^\star$, $t\in[0,1]$ such that $J_t$ is generic for all times $t\in[0,1]$.
We shall explain why a $J_0$-holomorphic curve curve persists under a generic deformation of the almost-complex structure. We note that for a second-category subset of $1$-parameter families of almost-complex structures, the deformation operator for all moduli spaces $\mathcal{M}_t$ is also surjective. This is the analogue of proposition \ref{prop2}. Moreover, the Fredholm index of a representative curve in $\mathcal{M}_t$ must be even, which precludes curves of negative index even in $1$-parameter families.


The above computation shows that 
\[ \mathcal{M}_{[0,1]} = \{ \left. (u,t)\right| u\in \mathcal{M}_t, \, t \in[0,1]\}. \]
is a one-dimensional manifold. Hence $\partial \mathcal{M}_{[0,1]}$ consists of a discrete set of points, which we call
\begin{equation}\label{cob} \partial \mathcal{M}_{[0,1]} \cong \mathcal{M}_0 \sqcup \mathcal{M}_1. \end{equation}

We noted above that for a second-category subset of $J_t$, the deformation operator for $\mathcal{M}_{[0,1]}$ will be surjective, and representative curves will have non-negative Fredholm index. Not every representative curve of Fredholm index zero will be cut out transversally, however, for all time $t \in [0,1]$. In this higher-dimensional context (dimension $\geq 6$), the projection map $\mathcal{M}_{[0,1]} \to [0,1]$ may not be a covering map, but it will give a cobordism from $\mathcal{M}_0$  to $\mathcal{M}_1$. The moduli space $\mathcal{M}_0$ was examined in the preceding propositions, and it was shown to be nonempty. Moreover, at time $t=0$ we required $Y$ to be holomorphically embedded with respect to $J_0$. Using (\ref{cob}), we regard $\mathcal{M}_{[0,1]}$ as a cobordism between $\mathcal{M}_0$ and $\mathcal{M}_1$.

We explain why $M_{[0,1]}$ is sequentially compact, which will imply that each $\mathcal{M}_t$ is compact.  To prove sequential compactness, consider a sequence $(u_k,t_k) \in \mathcal{M}_{[0,1]}$ and, after passing to a subsequence, assume that the $t_k$ converge to $t_\infty \in [0,1]$. Then, after possibly taking another subsequence of $(u_k,t_k)$, SFT compactness implies that the $u_k$ converge to a $J_{t_\infty}$-holomorphic building. Note that, by construction, this building contains levels in the completion of $X\setminus E$ and levels in the symplectization layer $\partial E \times \R$. We must show that this building is in fact a $J_{t_\infty}$-holomorphic plane. We do this by showing that any symplectization components must be trivial cylinders. Start with a curve $G$ in the lowest symplectization level. The results of proposition \ref{prop4} imply that $G$ must have positive and negative ends asymptotic to $\gamma_1$ only. By construction, we also have that the negative end of $G$ must be $\gamma_1^{(2d+1)}$. Let us say that $G$ has $s^+$ positive ends, with the $p^{th}$ positive end winding $a_p$ times around $\gamma_1$, and let us say that $G$ has $s^- =1$ negative end winding $2d+1$ times around $\gamma_1$. 
The number of positive ends of $G$ is at least the number of negative ends, both counted with multiplicity, which implies that
\[ \sum_{p=1}^{s^+}a_p \geq  2d+1.\]
Moreover 
\[  \left\lfloor \frac{a_p(\epsilon-\delta_1)}{(2d+1)(\epsilon-\delta_j)}\right\rfloor =  0,\]
whenever $2\leq j \leq n$ and whenever $a_p\leq 2d+1$ because of our choice of $\epsilon, \delta_1, \delta_2, \delta_3$, etc. We will omit these terms from the upcoming Fredholm index formula, knowing that they contribute positively or not at all to the index. 
The dimension of the elliptical families of Reeb orbits is zero, and in the symplectization layer the first Chern class is zero. Using the trivialization coming from the inclusion of the ellipsoid, we may use the Conley-Zehnder formula from lemma \ref{czgam}. The virtual Fredholm index formula gives
\begin{align*} Index(G) &= (n-3)(2 - s^+ - 1) + 2c_1(u) + \left(\sum_{p=1}^{s^+} CZ(\gamma_1^{(a_p)}) \right)-  CZ(\gamma_1^{(2d+1)})\\
&= (n-3)(1- s^+) + 0 + \left(\sum_{p=1}^{s^+} (2a_p+(n-1))\right) - (2(2d+1)+(n-1))\\
&=2(s^+-1) + 2\left(\sum_{p=1}^{s^+} a_p - (2d+1) \right)\geq 0\end{align*}
with equality if and only if $s^+=1$, and $\sum_{p=1}^{s^+}a_p =a_1= 2d+1$. This happens if and only if $G$ is a trivial cylinder.
By induction on levels, all curves in the symplectization layer have nonnegative Fredholm index, and if a symplectization component were not a trivial cylinder, then it must have strictly positive Fredholm index.
This would force a component in $X\setminus E$ to be (possibly a multiple cover of) a curve with negative Fredholm index. This configuration is precluded by genericity of $J \in \mathcal{J}^*$. Hence $M_{[0,1]}$ is compact

We claim that the cobordism $\mathcal{M}_{[0,1]}$ is not a null-cobordism. Then  $\mathcal{M}_0$ being nonempty implies $\mathcal{M}_1$ is nonempty.
We prove this claim by showing that the cardinality of $\mathcal{M}_0$, counted with orientation, is non-zero. Hence a null-cobordism cannot occur. This computation is done in propositions 10 and 11 of \cite{DcGH}, but is simpler here because the curves we consider have only a single negative end. A summary of the argument follows. Note that $X\setminus E$ admits an $S^1$ action that rotates the stabilized factor ($\mathbb{R}^2 \cong \C$ with coordinate $z_3$). Let us write $\mathcal{J}^*_{S^1}$ for the set of $S^1$ invariant and admissible almost-complex structures on $X\setminus E$ for which the moduli space $\mathcal{M}_{[0,1]}$ is nonempty. It is proved (in \cite{newobs} section 3.3.2) that such $S^1$-invariant almost-complex structures exist, making $\mathcal{J}^*_{S^1}$ nonempty. For generic $J \in \mathcal{J}^*_{S^1}$ all $J$-curves which meet an $S^1$ orbit exactly once will be regular. One then has automatic transversality for these curves. Once $\mathcal{J}^*_{S^1}$ is known to be nonempty, the assertion is that for some $J \in \mathcal{J}^*_{S^1}$, the curves in the corresponding moduli space $\mathcal{M}_0$ have image in the $4$-dimensional, holomorphic slice $Y$ (i.e. $z_3=0$). If not, some curve with image not contained entirely in the slice would meet the $S^1$ orbit more than once, hence in an $S^1$-family. Consider the projection of such a curve $u$ onto the slice $z_3=0$. The projection of this curve is non-injective and the curve $u$ multiply covers its projection. But such a curve $u$ can be excluded by index calculations.  

Finally, representatives of $\mathcal{M}_1$ persist under scaling the ellipsoid $E$.
Note that the set of ellipsoid embeddings into $X$ is connected, and one gets from one embedding of $E$ to any other embedding of $E$ through scaling and change of coordinate bases. Re-scaling the ellipsoid does not change the manifold $X\setminus E$ up to diffeomorphism. Furthermore, we can pull all structures (symplectic, almost-complex, etc.) back from the re-scaled $X\setminus E$ to the un-scaled $X\setminus E$. Re-scaling the ellipsoid has no effect on the Fredholm index of curves asymptotic to the ellipsoid, but it does affect the area of curves asymptotic to the ellipsoid. In the proof of this proposition, we have used only Fredholm arguments, making the proof agnostic of scale. 
This completes the proof of the proposition and the proof of Theorem  \ref{bigcurve}.
\end{proof}


\section{Higher Dimensions}\label{highdim}
We conclude by explaining how the above results extend to higher stabilizations, i.e. $n > 3$.  In section \ref{ell}, we should consider the $2n$-dimensional ellipsoid
\[ E=(\epsilon-\delta_1, (2d+1)(\epsilon-\delta_2),...,(2d+1)(\epsilon-\delta_n)),\]
where $\delta_2,...,\delta_n<\delta_1<\epsilon$ are small, and formula (\ref{oneone}) holds for $1\leq i \neq j \leq n$. This irrationality of the capacities of the ellipsoid implies that there are only $n$ closed Reeb orbits of $E$ along the coordinate planes. The eccentricity of the ellipsoid $E$ ensures that curves with positive area can only be asymptotic to the short orbit $\gamma_1$. (This is the first computaion of Proposition \ref{prop4}.)
 The formula for the Conley-Zehnder index of a Reeb orbit on a $2n$-dimensional ellipsoid is given in lemma \ref{czgam}.
 
 In section \ref{polyd}, the source polydisc $P(1,x)\times \R^{2n-4}$ should be replaced by 
\[ P(1,x,\underbrace{S,...,S}_{n-2}). \]
This polydisc contains a special subset $U$ defined as a toric domain by
\[ U = \mu^{-1}(\, (1-\epsilon,1)\times(x-(2d+1)\epsilon,x)\times(S/2,S)^{n-2} \,),\]
and the discussion of how to use a Hamiltonian to smooth $U$ to obtain $\Sigma$ generalizes in the obvious way. The formula for the Conley-Zehnder index of a Reeb orbit on $\Sigma$ in homology class $(k,\ell, m_1,...,m_{n-2})\in \Z^n\setminus \vec{0}$ generalizes to
\begin{equation}\label{higherCZ} CZ(c)= 2k + 2\ell + 2m_1+ ... + 2m_{n-2} +\frac{n-1}{2}.\end{equation}
We compactify the target polydisc $P(a,b)\times \R^{2n-4}$ as $X=\CP{1}(a)\times \CP{1}(b)\times \C^{n-2}$. The divisor at infinity becomes
\[ L_\infty = p_\infty \times \CP{1}(b)\times \C^{n-2} \cup\, \CP{1}(a)\times p_\infty \times \C^{n-2}, \]
and we assume $L_\infty$ is a $J$-holomorphic submanifold of $X$. The bidegree of a curve in $X$ is defined by its intersection number with the Poincar\'e dual of $L_\infty$.
We then assume that 
\[ E \to U \hookrightarrow P(a,b)\times \R^{2n-4} \hookrightarrow X.\]

The existence theorem for curves, Theorem \ref{bigcurve}, does not depend on $n$. The proof of this theorem is the whole of section \ref{curveproof}, and the Propositions \ref{prop1} through \ref{prop4} take place in four dimensions. It is not until Proposition \ref{prop5} that stabilization enters. The curve is stabilized by simply including it in the stabilized target manifold. The index computation in Proposition \ref{prop5} explicitly mentions $n$, and still gives a non-negative index when $n>3$. The conditions for $Index(G)=0$ are the same for $n>3$ in Proposition \ref{prop5}. The conclusion of the computation would be the same. In the penultimate paragraph of Proposition $5$, the $S^1$ action is replaced with a $\mathbb{T}^n$ action, which rotates each stabilized factor. The proof in the reference \cite{DcGH} goes through just the same.

Once the existence of a stabilized curve is ensured, the proof by contradiction of the main theorem can be done in any dimension. The formula (\ref{higherCZ}) will simplify, because we can show that all $m_j=0$ as we did in lemma \ref{emma}. Notice, in particular, that the Conley-Zehnder index (\ref{higherCZ}) increases by $1/2$ as $n$ increases by $1$. This contribution from the dimension is exactly canceled by the term that involves the dimension of the Reeb orbits in the virtual index formula, but only for a negatively asymptotic end. Specifically, we have shown for a negative end winding about a hyperbolic orbit $c$ on $\Sigma$ that
\[ CZ(c) - \frac{1}{2}\dim(c)= \left(2k + 2\ell + 2m_1+ ... + 2m_{n-2} +\frac{n-1}{2}\right) - \frac{1}{2}(n-1) = 2(k+\ell),\]
in all dimensions $n\geq 3$. Consequently, formula (\ref{virtind2}) and the formulas in lemma \ref{cinfpos} still describe the virtual Fredholm index of a plane in $X\setminus U$ with negative end(s) asymptotic to a Reeb orbit on $\Sigma$. For this reason, the results of lemma \ref{cinfpos} still hold for $n\geq 3$.
(A similar cancellation can be used to show that the so-called ``special curve'' has virtual Fredholm index zero, but this fact is never used.)  This computation essentially reduces the remainder of the proof to the $4$-dimensional case.

\bibliographystyle{plain}
\bibliography{polybib}
\end{document}